\documentclass[12pt]{amsart} 
\usepackage{amsmath,amsthm,amsfonts,amssymb,amscd,color}
\usepackage{graphicx,epsfig}
\newtheorem*{theorema}{Main Theorem}
\newtheorem*{corollary1}{Corollary}
\newtheorem{prop}{Proposition}[section]
\newtheorem*{USL}{Uniform Scale Lemma}
\newtheorem{lemma}[prop]{Lemma}
\newtheorem{theorem}[prop]{Theorem}

\theoremstyle{definition}
\newtheorem{remark}[prop]{Remark}

\renewcommand{\=}{ = }

\newcommand{\R}{\mathbb{R}}
\newcommand{\Z}{\mathbb{Z}}

\newcommand{\cL}{\mathcal{L}}
\newcommand{\cM}{\mathcal{M}}

\newcommand{\cQ}{\mathcal{Q}}
\newcommand{\cW}{\mathcal{W}}

\newcommand{\hJ}{\widehat{J}}
\newcommand{\hU}{\widehat{U}}

\DeclareMathOperator{\HD}{HD}
\DeclareMathOperator{\Crit}{Crit}
\DeclareMathOperator{\Cr}{Cr}
\DeclareMathOperator{\dist}{dist}
\DeclareMathOperator{\erg}{erg}

\DeclareMathOperator{\ac}{ac}
\DeclareMathOperator{\interior}{int}

\newcommand{\acip}{\mu_{\ac}}
\newcommand{\umu}{\underline{\mu}}
\newcommand{\un}{\underline{n}}

\newcommand{\Merg}{\cM_{\erg}(f)}

\newcommand{\hLambda}{\widehat{\Lambda}}
\newcommand{\hnu}{\widehat{\nu}}
\newcommand{\hmu}{\widehat{\mu}}

\begin{document}

\author{Yong Moo Chung}
\address{Department of Applied Mathematics,
Hiroshima University, Higashi-Hiroshima,
739-8527, Japan} 
\email{chung@amath.hiroshima-u.ac.jp}

\author{Juan Rivera-Letelier}
\address{Department of Mathematics, University of Rochester, Hylan
  Building, Rochester, NY~14627, U.S.A.} 
\email{riveraletelier@gmail.com}
\urladdr{\texttt{http://rivera-letelier.org}}

\author{Hiroki Takahasi}
\address{Keio Institute of Pure and Applied Sciences (KiPAS), Department of Mathematics,
Keio University, Yokohama,
223-8522, Japan} 
\email{hiroki@math.keio.ac.jp}
\urladdr{\texttt{http://www.math.keio.ac.jp/~hiroki/}}


\title[Large deviation principle in one-dimensional dynamics]{Large deviation principle in one-dimensional dynamics}
\date{\today}
\begin{abstract}
We study the dynamics of smooth interval maps with non-flat critical points.
For every such a map that is topologically exact, we establish the full (level-2) Large Deviation Principle for empirical means.
In particular, the Large Deviation Principle holds for every non\nobreakdash-renormalizable quadratic map.
This includes the maps without physical measure found by Hofbauer and Keller, 
and challenges the widely-shared view of 
the Large Deviation Principle as a refinement of laws of large numbers.
\end{abstract}

\maketitle

\section{Introduction}
An important concept in dynamical systems is that of \emph{physical measure}.
An invariant probability measure $\mu$ of a dynamical system $f$ is \emph{physical} if 
 there exists a set~$E$ of positive Lebesgue measure in the phase space such that
for every $x\in E$ the empirical mean on the orbit $\{x,f(x),f^2(x),\ldots,f^{n-1}(x)\}$
converges to~$\mu$ as $n\to\infty$, in the weak* topology.
The theory of large deviations aims to provide exponential bounds on the probability that the empirical means stay away from $\mu$.
See, \emph{e.g.}, \cite{DemZei98,Ell85} for general accounts of large deviation theory.

For uniformly hyperbolic diffeomorphisms, physical measures have been constructed in the pioneering works of 
Sina{\u{\i}}, Ruelle and Bowen \cite{Bow75, Rue76, Sin72}.
In this setting, the \emph{Large Deviation Principle} (\emph{LDP} for short) has been established by Takahashi
\cite{Tak84,Tak87}, Orey \& Pelikan \cite{OrePel89}, Kifer \cite{Kif90}, Young \cite{You90}; 
it describes stochastic features of deterministic dynamics with chaotic behavior.

In recent years there have been considerable efforts to extend these results beyond the uniformly hyperbolic setting.
All previous results we are aware of are restricted to maps satisfying a weak form of hyperbolicity, see for example~\cite{ChuTak12,ChuTak14,ComRiv11,Gri93,KelNow92,Li16,MelNic08,PrzRiv11,ReyYou08} and references therein.
The only ones establishing a full LDP are~\cite{ChuTak12} and~\cite[Theorem~B]{ChuTak14}, for a set of positive measure of quadratic maps satisfying the Collet-Eckmann condition~\cite{ColEck83}.
See also~\cite{ComRiv11,Gri93,Li16} for full LDPs for maps satisfying a weak form of hyperbolicity, in which the empirical measures are weighted with respect to an equilibrium state of a H{\"o}lder continuous potential.
In spite of the relative incompleteness of the theory, there was a belief among experts that the LDP holds under weaker assumptions. 

In this paper we study smooth interval maps with only non-flat critical points.
The presence of critical points is a severe obstruction to uniform hyperbolicity.
We establish a full level-2 LDP for every such map that is topologically exact.
In particular, the LDP holds for every non-renormalizable quadratic map.
Notably, this includes maps having no physical measure, like the quadratic maps found by Hofbauer \& Keller in~\cite{HofKel90,HofKel95}.
Notice that the formulation of the LDP \cite{DonVar75} does not {\it a priori} assume the strong law
of large numbers or the existence of a physical measure.

We now proceed to describe our main results in more detail.
\subsection{Statement of results}
\label{ss:main}
Throughout this paper we set $X=[0,1]$,
and for a measurable subset~$A$ of~$X$ we denote by~$|A|$ its Lebesgue measure.

A \emph{critical point} of a differentiable map $f\colon X\to X$ is a
point at which the derivative of~$f$ vanishes.
Denote by~$\Crit(f)$ the set of critical points of~$f$.
A critical point~$c$ of~$f$ is \emph{non\nobreakdash-flat} if there are~$\ell > 1$
and diffeomorphisms~$\phi$ and~$\psi$ of~$\mathbb R$ such that $\phi(c)=\psi(f(c))=0$ 
 and such that for every~$x$ in a neighborhood of~$c$,
\begin{equation*}
|\psi \circ f(x) |= |\phi(x)|^{\ell}.
\end{equation*}
Note that a continuously differentiable map with only non-flat critical
points has
at most a finite number of critical points.

Denote by $\mathcal M$ the space of Borel probability measures on~$X$ endowed with the weak* topology.
For $x\in X$ denote by $\delta_x\in\mathcal M$ the Dirac measure at~$x$.
Given a continuous map~$f \colon X \to X$ and an integer~$n \ge 1$, define
$\delta_x^n= \frac{1}{n} \sum_{i=0}^{n-1}\delta_{f^i(x)}.$
The map~$f$ is \emph{topologically exact}
if for every nonempty open subset~$U$ of~$X$ there is an integer $n
\ge 1$ such that $f^n(U)=X$. 

\begin{theorema}
Let $f\colon X\to X$ have H{\"o}lder continuous derivative and only
non-flat critical points.
If~$f$ is topologically exact, then the full level-2 Large Deviation Principle holds, namely,
there exists a lower semi-continuous function $I\colon\mathcal M\to[0, \infty]$ such that:

-(lower bound) for every open subset $\mathcal G$ of $\mathcal M$,
\begin{equation*}\label{low}
\liminf_{n\to\infty}\frac{1}{n} \log \left|\left\{x\in X\colon\delta_x^n\in\mathcal G\right\}\right|\geq-\inf_{\mathcal G}I;\end{equation*} 

-(upper bound) for every closed subset $\mathcal K$ of $\mathcal M$,
\begin{equation*}\label{up}
\limsup_{n\to\infty}\frac{1}{n} \log \left|\left\{x\in X\colon\delta_x^n\in\mathcal K\right\}\right|\leq-\inf_{\mathcal K}I.\end{equation*} 
\end{theorema}
In the theorem above and in the rest of the paper,
$$ \log 0=-\infty,
\inf\emptyset=\infty
\text{ and }
\sup\emptyset=-\infty. $$

The function~$I$ is called a \emph{rate function}.
From the general theory of large deviations~\cite{DemZei98,Ell85}, the LDP determines~$I$ uniquely.
We show that~$-I$ is the upper semi-continuous regularization of the ``free energy function''.
Then the rate function is convex, and it is characterized as the Legendre transform of the cumulant generating function, see Sect.\ref{ss:comments}.

The traditional application of the LDP in dynamical systems is for maps having a physical measure.
In the probabilistic viewpoint of dynamical systems, the existence of a physical measure is analogous to the law of large numbers, and the LDP is a refinement of this law.
For concreteness, consider a map~$f \colon X \to X$ as in the Main Theorem that in addition has a physical measure~$\mu$.
Then the rate function~$I$ vanishes at~$\mu$ and, assuming~$f$ is sufficiently regular, for Lebesgue almost every point~$x$ in~$X$ the sequence of empirical measures~$\{ \delta_x^n \}_{n = 1}^{\infty}$ converges to~$\mu$ in the weak* topology, see~\cite[Theorem~8]{CaiLi09}.
This last property is thus analogous to the law of large numbers, and the LDP given by the Main Theorem is a refinement: the speed of convergence is controlled by the rate funcion~$I$.

The LDP given by the Main Theorem applies to situations beyond the traditional one, since it does not require the existence of a physical measure.
Note also that the LDP in the Main Theorem does not require any weak form of hyperbolicity.
To illustrate the broader applicability of the the Main Theorem, we give two new insights into the dynamics of quadratic maps.
The first concerns one of the quadratic maps~$f_0$ without physical measures studied by Hofbauer \& Keller in~\cite{HofKel90, HofKel95}.
The rate function of~$f_0$ vanishes entirely on its effective domain, in sharp contrast with the uniformly hyperbolic case where the rate function only vanishes at the physical measure.
The LDP given by the Main Theorem gives a quantitative version of the ``maximal oscillation'' property studied by Hofbauer \& Keller in~\cite{HofKel95}, see Sect.\ref{ss:comments} for details.
We also consider the quadratic Fibonacci map~$f_*$ studied by Lyubich \& Milnor~\cite{LyuMil93}, Keller \& Nowicki~\cite{KelNow95}, and others.
The equilibrium states of~$f_*$ for the geometric potential~$- \log |Df_*|$ form a segment, having the physical measure~$\mu_*$ of~$f_*$ as an endpoint.
Although the basin of an equilibirum state~$\mu$ different from~$\mu_*$ has zero Lebesgue measure, the LDP given by the Main Theorem implies that~$\mu$ still attracts a significant set of initial conditions, see Sect.\ref{ss:comments} for details.

Besides the uniformly hyperbolic case mentioned at the beginning of the introduction, the only previous full LDPs were established in~\cite{ChuTak12} and~\cite[Theorem~B]{ChuTak14} for a set of positive measure of quadratic maps satisfying the Collet-Eckmann condition.
See also~ \cite{ComRiv11, Gri93, Li16}\footnote{See also the survey article of Denker~\cite{Den96}.} for full LDPs for maps satisfying a weak form of hyperbolicity, in which the empirical measures are weighted with respect to an equilibrium state of a H{\"o}lder continuous potential.
For local LDPs, see~\cite[Theorems~1.2 and~1.3]{KelNow92}, \cite{MelNic08}, \cite[Corollary~B.4]{PrzRiv11}, \cite{ReyYou08}, and references therein.

We now state a corollary of the Main Theorem that follows from the general theory of large deviations.
We use it below to compare our result with previous related ones.
Let~$\mathcal M(f)$ be the subspace of~$\mathcal M$ of those measures that are $f$-invariant.
For a continuous function $\varphi\colon X\to\mathbb R$ define
$$ c_\varphi
=
\min \left\{ \int \varphi d\nu\colon\nu\in\mathcal M(f) \right\}
\text{ and }
d_\varphi
=
\max \left\{ \int \varphi d\nu\colon\nu\in\mathcal M(f) \right\}, $$
and for each integer $n\geq1$ and~$x$ in~$X$ write
$$ S_n\varphi(x)
=
\sum_{i=0}^{n-1}\varphi\circ f^i(x)
=
n \int \varphi d \delta_x^n. $$
Moreover, define a rate function $q_\varphi\colon \mathbb R\mapsto [0,  \infty]$ by
\begin{displaymath}
q_\varphi(t)
=
\inf\left\{I(\mu)\colon\mu\in\mathcal M, \int\varphi d\mu=t\right\}.
\end{displaymath}
This function is bounded on~$[c_\varphi, d_\varphi]$ and constant equal to~$\infty$ on~$\R \setminus [c_\varphi,d_\varphi]$.
Furthermore, $q_\varphi$ is convex on~$\R$, and therefore continuous on~$(c_\varphi,d_\varphi)$.

The following corollary is a direct consequence of the Main Theorem and of the contraction principle, see for example~\cite{DemZei98,Ell85}.
\begin{corollary1}
Let $f\colon X\to X$ have H{\"o}lder continuous derivative and only
non-flat critical points.
If~$f$ is topologically exact, then for every continuous function $\varphi\colon X\to\mathbb R$ satisfying $c_\varphi<d_\varphi$ and for every interval~$J$ intersecting $(c_\varphi,d_\varphi)$,
$$ \lim_{n\to\infty}\frac{1}{n}\log \left|\left\{x\in X\colon \frac{1}{n}S_n\varphi(x) \in J\right\}\right|
=
-\inf_{ J}q_\varphi.$$
\end{corollary1}

One previous result relevant to this corollary is that of Keller \& Nowicki
 \cite[Theorem~1.2]{KelNow92}, in the case where~$f$ is a $S$-unimodal map satisfying the Collet-Eckmann condition, see the definition of $S$-unimodal map below.
Denoting by~$\acip$ the unique absolutely continuous invariant probability (\emph{acip} for short) of~$f$, they proved that the corollary holds with~$\varphi = \log |Df|$ for every interval~$J$ whose boundary is contained in a small neighborhood of~$t = \int \log |Df| d\acip$.
 
Let us illustrate a broad applicability of the Main Theorem and its corollary in the context of ``$S$-unimodal'' maps, which we proceed to recall.
A non-injective continuously differentiable map $f \colon X \to X$ is \emph{unimodal}, if~$f(\partial X) \subset \partial X$, and if~$f$ has a unique critical point.
The unique critical point~$c$ of such a map must be in the interior of~$X$ and be of ``turning'' type; that is, $f$ is not locally injective at~$c$.
The map~$f$ is \emph{$S$\nobreakdash-unimodal}, if in addition~$c$ is non-flat for~$f$, and if~$f$ is of class~$C^3$ and has negative Schwarzian derivative on~$X \setminus \{ c \}$; in this context the non-flatness condition is the same as above with the additional requirement that the diffeomorphisms~$\phi$ and~$\psi$ are of class~$C^3$.

Each $S$-unimodal map has exactly one of the following dynamical characteristics:
\begin{enumerate}
\item[(i)]
it has an attracting cycle;
\item[(ii)]
it is infinitely renormalizable;
\item[(iii)]
it is at most finitely renormalizable.
\end{enumerate}
In case (iii) there is an integer~$p \ge 1$ and a closed interval~$J$ containing the critical point of~$f$ in its interior, such that $f^p(J)\subset J$, such that the return map $f^p\colon J\to J$ is topologically exact, and such that the intervals~$J$, $f(J)$, \ldots, $f^{p - 1}(J)$ have mutually disjoint interiors, see for example the combination of \cite[Theorem~V.1.3]{dMevSt93} and \cite[Theorem~2.19 and Proposition~2.34]{Rue17}.
This implies that a rescaling of~$f^p|_J$ satisfies the assumptions of the Main Theorem.
It follows that \emph{the LDP holds for \emph{every} at most finitely renormalizable $S$-unimodal map.}

For a real analytic family of $S$-unimodal maps with quadratic critical point and non-constant combinatorics, such as the quadratic family,
Lebesgue almost every parameter corresponds to either case~(i) or case~(iii), and in the latter case there is an acip~\cite{AviLyudMe03,Lyu02}.
The set of parameters corresponding to acips has positive Lebesgue measure~\cite{BenCar85,Jak81}.

\subsection{Further results and comments}
\label{ss:comments}
We characterize the rate function~$I$ in the Main Theorem as follows.
For $\nu\in\mathcal M(f)$ denote by $h(\nu)$ the entropy of $\nu$, and
define the \emph{Lyapunov exponent} $\lambda(\nu)$ of $\nu$ by $\lambda(\nu)= \int\log|Df|d\nu$.
The \emph{free energy function $F \colon \cM \to [-\infty,  \infty)$} is defined by,
$$ F(\nu)
=
\begin{cases}h(\nu)-\lambda(\nu) &\text{ if }\nu\in\mathcal M(f);\\
-\infty&\text{ otherwise.}\end{cases}$$
Since the map~$f$ in the Main Theorem is topologically exact, it has the specification property. Then it has no hyperbolic attracting periodic point 
and empirical measures along periodic orbits are dense in the space of invariant measures \cite[Theorem~1]{Sig74}.
Together with the upper semi-continuity of the Lyapunov exponent, this implies that for every~$\nu \in \cM(f)$ we have~$\lambda(\nu)\geq0$, see also
\cite[Proposition~A.1]{Riv1204}.
We show that the rate function~$I$ in the Main Theorem is given by
\begin{equation}
  \label{eq:6}
  I(\mu)
=
-\inf_{\mathcal G \ni \mu}\sup_{\mathcal G}F,
\end{equation}
where the infimum is taken over all open subsets~$\mathcal G$ of~$\mathcal M$ containing~$\mu$.
It follows that~$I$ is convex, and therefore that~$I$ is the Legendre transform of the cumulant generating function, see for example~\cite[Theorem~4.5.10(b)]{DemZei98}.
On the other hand, using~\eqref{eq:6} and the fact that the rate function takes only nonnegative values, we obtain from the LDP in the Main Theorem that for every~$\nu\in\mathcal M(f)$ we have~$F(\nu) \le 0$.
  This is known as Ruelle's inequality \cite{Rue78}.
  Note also that the rate function vanishes at each equilibrium state of~$f$ for the geometric potential $-\log |Df|$.
  That is, the rate function~$I$ vanishes at every measure~$\nu \in \cM(f)$ for which Rohlin's formula $F(\nu)=0$ holds.
  See below for an example where the function vanishes at a measure that is not an equilibrium state.
  
Consider a $S$-unimodal map~$f$ with a non-flat critical point that satisfies the Collet-Eckmann condition \cite{ColEck83}.
Then the corresponding rate function vanishes precisely at the (unique) acip~\cite[Theorem~A.1]{ChuTak17}.
As mentioned earlier, for such a map~$f$ we have the traditional application of the LDP in the Main Theorem as a refinement of the law of large numbers.

We now describe two applications of the LDP in the Main Theorem that go beyond the traditional application of refining the law of large numbers.
First, we consider one of the quadratic maps~$f_0 \colon X \to X$ without physical measures studied by Hofbauer \& Keller in~\cite[Theorem~5]{HofKel90} and~\cite{HofKel95}, see Theorem~\ref{HK} in the Appendix for a precise description.
The Main Theorem applies to~$f_0$ and the corresponding rate function vanishes entirely on its effective domain, see Theorem~\ref{HKrate} in the Appendix.
This is in sharp contrast with the uniformly hyperbolic case, for which the rate function only vanishes at the physical measure.
Applying the Corollary of the Main Theorem to~$f_0$, we obtain:
\begin{quote}
  Choose~$\varepsilon > 0$, an arbitrary invariant measure~$\mu$, and an arbitrary continuous function~$\varphi \colon X \to \R$.
  Then for~$n \ge 1$, the set~$E_n$ of all the initial conditions~$x_0$ for which
  \begin{displaymath} 
    \left| \frac{1}{n} \sum_{j = 0}^{n-1} \varphi(f_0^j(x_0)) - \int \! \varphi \, d \mu \right|
    \le \varepsilon,
  \end{displaymath}
  is sub-exponentially large with respect to~$n$:
  \begin{equation}
    \label{eq:4}
    \lim_{n \to \infty} \frac{1}{n} \log |E_n| = 0.
  \end{equation}
\end{quote}
Equivalently, there is a sub-exponentially large set of initial conditions for which the Birkhoff average of~$\varphi$ is near the mean with respect to~$\mu$.
This happens simultaneously for every invariant measure~$\mu$, and gives a quantitative version of the ``maximal oscillation" property of~$f_0$ shown by Hofbauer \& Keller in~\cite{HofKel95}.

The second application is for the Fibonacci quadratic map~$f_* \colon X \to X$, studied by Lyubich \& Milnor~\cite{LyuMil93}, Keller \& Nowicki~\cite{KelNow95}, and others.
This map has a physical measure~$\mu_*$ whose basin of attraction has full Lebesgue measure on~$X$~\cite[Theorem~1.3(4)]{LyuMil93}.
That is, for Lebesgue almost every point~$x$ in~$X$ the sequence of empirical measures~$\{ \delta_x^n \}_{n = 1}^{\infty}$ converges to~$\mu$ in the weak* topology. 
On the other hand, the closure of the critical orbit is a Cantor set that supports a unique invariant probability measure~$\nu_*$~\cite[Theorem~1.2]{LyuMil93}.
The measures~$\mu_*$ and~$\nu_*$ are the unique ergodic equilibrium states of~$f_*$ for the geometric potential $- \log |Df_*|$, so every equilibrium state is a convex combination of~$\mu_*$ and~$\nu_*$~\cite[Corollary~3.11 and Example~3.13]{BruKel98}.
The Main Theorem applies to~$f_*$ because this map is non-renormalizable.
The rate function~$I$ thus vanishes at each convex combination of~$\mu_*$ and~$\nu_*$.
Moreover, $I$ can only vanish at the convex combinations of~$\mu_*$ and~$\nu_*$, because the free energy function~$F$ for $f_*$ is upper semi-continuous and therefore~$I = -F$~\cite[Corollary~2.6 and Proposition~2.9]{BruKel98}.
Consider an equilibrium state~$\mu$ different from the physical measure~$\mu_*$.
Since~$\mu \neq \mu_*$, the basin of~$\mu$ has zero Lebesgue measure.
Nevertheless, $I(\mu) = 0$ and therefore the LDP lower bound given by the Main Theorem shows that~$\mu$ does attract a significant set of initial conditions: for every~$n \ge 1$ the set~$E_n$ of initial conditions~$x_0$ for which the empirical mean~$\delta_{x_0}^n$ is close to~$\mu$ satisfies~\eqref{eq:4}.
That is, $E_n$ is sub-exponentially large with~$n$.
Furthermore, the LDP given by the Main Theorem also shows that the equilibrium states of~$f_*$ for the potential~$- \log |Df_*|$ are the only invariant measures satisfying this property.
There is an analogous application of the LDP for Manneville-Pomeau maps, see~\cite{PolShaYur98}, \cite[Section~5]{Chu11} and \cite[Appendix~B]{ChuTak17}.
For a certain range of parameters, there is a physical measure whose basin has full Lebesgue measure, and the rate function vanishes precisely at the convex combinations of this measure and the Dirac mass at the indifferent fixed point.

Usually the free energy function~$F$ is not upper semi-continuous,\footnote{Although the entropy map is upper
  semi-continuous as a function of measures, the Lyapunov exponent
  function is not lower semi-continuous in general
since $f$ has critical points, see for
example~\cite[Proposition~2.8]{BruKel98}.} so in
general~$I$ is different from~$-F$.
For a concrete example for which these functions differ,
consider the quadratic map $f(x)=4x(1-x)$. Then $0$ is a hyperbolic repelling fixed point
and $F(\delta_0)=-\log4$. The Lyapunov exponents of all other ergodic measures are $\log2$,
and $\delta_0$ is weak*-approximated by measures supported on periodic points, and so 
$I(\delta_0) = \log2$.
For another example, 
consider a quadratic map~$f_1$ given by \cite[Theorem~3]{HofKel90}, whose unique physical measure is the Dirac measure supported at a repelling fixed point~$p$ of~$f_1$.
As mentioned before $I(\delta_p) = 0$, but~$F(\delta_p) = - \log |Df_1(p)| < 0$.
This is also an example where the rate function vanishes at a measure that is not an equilibrium state.

In~\cite{ChuTak12} a full level-2 LDP similar to the Main Theorem is shown for a positive measure set of Collet-Eckmann quadratic maps.
In this result, the rate function is the same as in the Main Theorem, but instead of weighting the empirical measures with respect to the Lebesgue measure, in~\cite{ChuTak12} they are measured with respect to the acip.
Combining both of these LDPs, we obtain that the Lebesgue measure and the acip are sub-exponentially close on a large class of dynamically defined sets.
It is not clear to us whether the LDP in~\cite{ChuTak12} holds for every Collet-Eckmann quadratic map, or if a parameter exclusion as in~\cite{ChuTak12} is needed.

Our methods apply with minor modifications to complex rational maps that are ``backward stable'' in the sense of~\cite{BloOve04a,Lev98a}; this is a condition analogous to the conclusion of Lemma~\ref{upbound}.
There is a large class of rational maps satisfying this property, including every polynomial with locally connected Julia set and all cycles repelling, see~\cite[Corollary~1]{Lev98a}.
There are however quadratic maps with all cycles repelling that are not backward stable, see~\cite[Remark~2]{Lev98a}.
Furthermore, it is not known whether every rational map satisfies the specification property, or some of this consequences, like the results in \cite{Sig74}.

\subsection{Outline of the paper}
\label{ss:outline}
In this section we outline the proof of the Main Theorem, and simultaneously describe the organization of the paper.

The proof of the Main Theorem follows the strategy originated in~\cite{Chu11} and that has been developed in~\cite{ChuTak12,ChuTak14}.
The main new ingredient is a diffeomorphic pull-back argument that simplifies the construction substantially, and that allows us to apply it to a larger class of maps.
The proof is divided in two parts: the lower bound is shown in Sect.\ref{s:lower bound}, and the upper bound in Sects.\ref{fundamental} and~\ref{s:upper bound}.

We show that the lower bound holds without the non-flatness hypothesis.
Roughly speaking, the proof of the lower bound consists of finding a set of points whose empirical means are close to a given invariant measure.
In the case this last measure is hyperbolic, the desired set is easily found using Katok-Pesin theory, which allows one to approximate each hyperbolic measure by hyperbolic sets in a particular sense.
The main difficulty is to deal with non-hyperbolic measures.
We use the specification property to approximate a non-hyperbolic measure by hyperbolic measures, in a suitable sense.
In this way we reduce the case of non-hyperbolic measures to the case of hyperbolic measures.

The upper bound is much harder, because a global control of the dynamics is required.
The main idea is to construct certain horseshoes with a finite number of branches that are tailored to a given open subset of~$\mathcal{M}$. 
This construction is necessarily involved due to the presence of the critical points.
In~\cite{ChuTak12,ChuTak14}, this method was implemented under strong assumptions on the orbit of the critical value, as mentioned earlier in the introduction.
In this paper, we use a diffeomorphic pull-back argument to replace the analytic horseshoe constructions in~\cite{ChuTak12,ChuTak14} by one of 
more topological flavor, enabling us to dispense with the strong assumptions on the critical orbits altogether.

The diffeomorphic pull-back argument is developed in Sect.\ref{fundamental}, where it is stated as the ``Uniform Scale Lemma.''
One of the main ingredients in the proof of this lemma are some general sub-exponential distortion bounds (Proposition~\ref{p:sub-exponential ratio growth} in~Sect.\ref{size}.)
These sub-exponential distortion bounds are combined with a method that goes back to~\cite{PrzRivSmi03}, to carefully avoid critical points and choose diffeomorphic
pull-backs.
The preliminary results needed to implement this method are established in Sect.\ref{dimension}, and the proof of the Uniform Scale Lemma is given in Sect.\ref{ss:proof of USL}.

The proof of the upper bounds is completed in Sect.\ref{s:upper bound}.
The main step is to construct, for a given basic open set of~$\cM(f)$ and for each large integer $n\geq1$, a certain horseshoe with inducing time~$q$, where $q\geq n$ and $ q-n =o(n)$ as $n\to\infty$ (Proposition~\ref{horseshoe} in Sect.\ref{covering}.)
By a {\it horseshoe with inducing time $q$} we mean a finite collection $L_1,L_2,\ldots,L_t$ of pairwise disjoint closed intervals
such that $f^q$ maps each $L_i$, $i\in\{1,2,\ldots,t\} ,$ diffeomorphically
onto an interval whose interior contains $\bigcup_{i=1}^tL_i$.
The inducing time~$q$ consists of three explicit parts: in the first~$n$ iterations, the intervals are mapped to a ball of radius~$n^{-\alpha}$, for a fixed constant~$\alpha>1$, centered at a carefully chosen base point;
in the second part, of roughly~$\log n$ iterations, intervals reach a fixed scale $\kappa>0$
independent of~$n$;
the third part, of a bounded number of iterations, the intervals return to a prefixed small interval.
In order to reach the scale $\kappa$, a key ingredient is the Uniform Scale Lemma in Sect.\ref{fundamental}.
Once the horseshoe is constructed, we prove two intermediate estimates in Sect.\ref{intermediate}.
The first is restricted to a small interval (Proposition~\ref{upper0}), and the second is a global estimate
(Proposition~\ref{upper}) obtained by using topological exactness to spread out the local estimate.
The local estimate is used to treat inflection critical points.
The proof of the upper bound is completed in Sect.\ref{end}.

\subsection{Notation}
\label{constants}
The following notation and terms are used in the rest of the paper.
For $x\in X$ and $\eta>0$ denote by $B(x,\eta)$ the closed ball of
radius~$\eta$ centered at~$x$, \emph{i.e.},
$$B(x,\eta)=\{y\in X\colon |y-x|\leq \eta\},$$
and for subsets~$A$ and~$A'$ of~$X$ define 
$$ B(A,\eta) \= \bigcup_{x\in A}B(x,\eta), \quad
\dist(x, A) \= \inf \{ |x - a| \colon a \in A \}, $$
and
$$\dist(A, A') \= \inf \{ |a - a'| \colon a \in A, a' \in A' \}. $$
A subset $F$ of $X$ is called \emph{$\eta$-dense} if 
$B(F,\eta) = X$ holds.
For a subset~$A$ of~$X$, denote by~$\HD(A)$ the Hausdorff dimension of~$A$.

Let~$f \colon X \to X$ be continuously differentiable.
A subset~$K$ of~$X$ is \emph{forward $f$-invariant} if $f(K)\subset K$.
The set~$K$ is called \emph{hyperbolic}, if there exist $C>0$ and $\lambda>1$ such that for every $x\in K$ and every integer $n\geq1$, $|Df^n(x)|\geq C\lambda^n$ holds.

\section{Large deviations lower bound}
\label{s:lower bound}
In this section we prove the large deviations lower bound in the Main Theorem.
As the proof below shows, these estimates hold without the non-flatness hypothesis.  
The following is the key estimate and it contains Ruelle's inequality.
It must be noted that in the following estimate we have to treat
measures with zero Lyapunov exponent.

\begin{prop}[Key Estimate]
\label{lower-prop}
Let $f \colon X \to X$ have H{\"o}lder continuous derivative and at most a finite
number of critical points.
Assume~$f$ is topologically exact.
Let $l\geq1$ be an integer, $\varphi_1,\varphi_2,\ldots,\varphi_l\colon X\to\mathbb R$ continuous
functions and $\alpha_1,\alpha_2,\ldots,\alpha_l \in  \mathbb R$.
Then for every $\mu\in\mathcal M(f)$ such that
$\int\! \varphi_jd\mu > \alpha_j$ for every $j\in\{1,2,\ldots,l\}$,
\begin{equation*}
 \liminf_{n\to\infty} \frac{1}{n} \log \left| \left\{ x\in X\colon
\frac{1}{n}S_n\varphi_j(x)>\alpha_j \text{ for every $j\in\{1,2,\ldots,l\}$}\right\} \right|
\geq F(\mu).
\end{equation*}
\end{prop}

In the proof of this proposition we use the following version of
Katok's theorem, which allows one to approximate each hyperbolic
measure by hyperbolic sets in a particular sense, compare with~\cite[Theorem~S.5.9]{KatHas95} and~\cite[Theorem~4.1]{PrzRiv19}.
Using Dobbs' adaptation of Pesin's theory to interval maps~\cite[Theorem~6]{Dob14}, the
proof is a slight modification of that of~\cite[Theorem~S.5.9]{KatHas95}
and hence we omit it.
For a continuous map~$f \colon X \to X$, a subset~$U$ of~$X$, and an integer $n\geq1$, each connected component of $f^{-n}(U)$ is called a \emph{pull-back of $U$ by $f^n$}. 
If in addition~$f$ is differentiable, then a pull-back~$J$ of~$U$ by~$f^n$ is called \emph{diffeomorphic} if $f^n
\colon J \to U$ is a diffeomorphism.
\begin{lemma}
\label{katok}
Let $f \colon X \to X$ have H{\"o}lder continuous derivative and at most a finite
number of critical points.
Let $\mu\in\mathcal M(f)$ be ergodic and such that $h(\mu)>0.$
Let $l\ge 1$ be an integer, and $\varphi_1, \ldots , \varphi_l \colon X\to \mathbb R$ continuous functions.
Then for every $\varepsilon>0$ there are integers $k\geq 2$ and
$m\geq1$ satisfying $\frac{1}{m} \log k  \ge h(\mu) -\varepsilon$, a
closed subinterval~$K$ of~$X$, and pairwise disjoint diffeomorphic pull-backs $K_1,
\ldots , K_k$ of~$K$ by~$f^m$ contained in~$K$, such that the following holds:
$$\left| \frac{1}{m}S_m \varphi_j (x) -  \int\! \varphi_j d\mu \right| < \varepsilon\ \text{for every $x\in\bigcup_{i=1}^k K_i$ and every $j\in\{1,\ldots,l\}$};$$
and
$$e^{(\lambda (\mu) -\varepsilon )m }\leq|Df^m(x) | \le e^{(\lambda (\mu) +\varepsilon )m}\    \text{for every $x\in\bigcup_{i=1}^k K_i$}.$$
\end{lemma}

\begin{proof}[Proof of Proposition~\ref{lower-prop}]
Fix~$\varepsilon > 0$ sufficiently small so that 
$\int \varphi_j d\mu > \alpha_j + \varepsilon$ holds for every $j\in\{1,
\ldots, l \}$.
For each~$(n_0, \ldots, n_{l + 1})\in\Z^{l + 2}$ put
$$ C((n_0, \ldots, n_{l + 1}))
\=
\left[ n_0 \frac{\varepsilon}{3}, (n_0 + 1) \frac{\varepsilon}{3}
\right) \times \cdots \times \left[ n_{l + 1} \frac{\varepsilon}{3}, (n_{l + 1} + 1) \frac{\varepsilon}{3}
\right). $$
Denote by~$\Merg$ the subset of~$\cM(f)$ of ergodic measures, and let~$\Phi \colon \Merg \to \R^{l + 2}$ be the function defined by
$$ \Phi(\nu)
\=
\left( h(\nu), \chi(\nu), \int\! \varphi_1 d\nu, \ldots,
  \int\! \varphi_l d\nu \right). $$
Finally, let~$Z$ be the subset of~$\Z^{l + 2}$ of those~$\un$ such
that~$\Phi^{-1}(C(\un))$ is nonempty, set~$s \= \# Z$, choose a bijection~$\iota \colon
\{1, \ldots, s \} \to Z$, and for each~$i\in\{1, \ldots, s \}$ choose a measure~$\mu_i \in  \Merg$ in~$\Phi^{-1}(C(\iota(i)))$.
Thus, if~$\umu$ is the unique probability measure on~$\Merg$ such that~$\mu
= \int \nu d \umu(\nu)$, and for each~$i\in\{1, \ldots, s \}$ we
put~$\beta_i \= \umu \left( \Phi^{-1}(C(\iota(i))) \right)$, then the
measure $\mu' \= \beta_1 \mu_1 + \cdots + \beta_s \mu_s$ is
in~$\cM(f)$, and satisfies
$|h(\mu) - h(\mu')| \le \frac{\varepsilon}{3}$, $|\lambda(\mu) - \lambda(\mu')|  \le \frac{\varepsilon}{3}$, and for each~$j\in\{1, \ldots, l \}$,
$$ \left| \int\! \varphi_j d \mu - \int\! \varphi_j d\mu' \right|  \le
\frac{\varepsilon}{3}. $$

For each~$i\in\{1, \ldots, s \}$ define integers~$k_i$ and~$m_i$
and subintervals~$K^i$, $K^i_1$, \ldots, $K^i_{k_i}$ of~$X$, as
follows.
In the case where~$h(\mu_i) > 0$, let $k_i \= k$, $m_i \= m$, $K^i \= K,
K_1^i \= K_1, \ldots, K^i_{k_i} \= K_{k_i}$ be as in Lemma~\ref{katok} with~$\varepsilon$ replaced by~$\frac{\varepsilon}{3}$.
Suppose~$h(\mu_i) = 0$.
By \cite[Theorem~1]{Sig74} and the upper semi-continuity of the Lyapunov exponent function there is a periodic point~$p$ such that, if we denote
by~$N \ge 1$ its minimal period, then~$\frac{1}{N} \log |Df^N(p)|
\le \lambda(\mu_i) + \frac{\varepsilon}{6}$ and for each~$j\in\{1, \ldots, l \}$,
$$ \left| \frac{1}{N} S_N \varphi_j(p) - \int\! \varphi_j d\mu_i \right|
\le
\frac{\varepsilon}{6}. $$
Using that~$f$ is topologically exact, it follows that for every sufficiently
small interval~$K$ containing~$p$, the pull-back~$K_1$ of~$K$ by~$f^N$
containing~$p$ is contained in~$K$.
Reduce~$K$ if necessary, so that~$f^N \colon K_1 \to K$ is a
diffeomorphism, and such that for every~$x\in K_1$ we have
$\frac{1}{N} \log |Df^N(x)| \le \lambda(\mu_i) +
\frac{\varepsilon}{3}$ and for each~$j\in\{1, \ldots, l\}$,
$$ \left| \frac{1}{N} S_N \varphi_j(x) - \int\! \varphi_j d\mu_i
\right|
\le
\frac{\varepsilon}{3}. $$
Set~$k_i \= 1$, $m_i \= N$, $K^i \= K$, and~$K^i_1 \= K_1$.

Take an integer $M \ge 1$ such that for each~$i\in\{1, \ldots, s\}$ we
have~$f^{M} (K^i) = X$, and fix an integer $n \ge 1$.
For each~$i\in\{1, \ldots, s\}$, put
$$ \ell_i \= \left[ \frac{\beta_i n}{m_i} \right]
\text{ and }
n_i \= \ell_i m_i + M, $$
and denote by~$\cL_i$ the collection connected components of~$\left(
  f^{m_i}|_{K^i_1 \cup \cdots \cup K^i_{s(i)}} \right)^{-\ell_i}(K^i)$.
Note that~$\# \cL_i = k_i^{\ell_i}$, and that for each~$L\in\cL_i$ we have~$f^{n_i}(L) = X$.
Furthermore, for each~$x\in L$ we have
\begin{equation}
  \label{eq:1}
  \frac{1}{n_i} \log |Df^{n_i}(x)|
\le
\frac{\ell_i m_i}{n_i} \left(\lambda(\mu_i) +
  \frac{\varepsilon}{3}\right) + \frac{M}{n_i} \log \left(\sup_X|Df|\right),
\end{equation}
and for each~$j\in\{1, \ldots, l \}$ we have
\begin{equation}
  \label{eq:2}
  \left| \frac{1}{n_i} S_{n_i} \varphi_j(x) - \int\! \varphi_j
  d\mu_i \right|
\le \frac{\ell_i m_i}{n_i} \frac{\varepsilon}{3} + \frac{M}{n_i} \sup_X
|\varphi_j|.
\end{equation}

Set~$m \= n_1 + \cdots + n_s$, and note that the sets in
$$ \cL
\=
\left\{ \left(f^{n_1}|_{L_1} \right)^{-1} \circ \cdots \circ \left(f^{n_s}|_{L_s}\right)^{-1}(X) \colon L_1 \in \cL_1, \ldots,
    L_s \in \cL_s \right\} $$
are pairwise disjoint, and that each set in~$\cL$ is mapped
onto~$X$ by~$f^m$.
On the other hand, if~$n$ is sufficiently large, then
\begin{multline*}
\frac{1}{m} \log (\# \cL)
=
\frac{\ell_1 \log k_1 + \cdots + \ell_s \log k_s}{n_1 + \cdots +
  n_s}
\ge
\left( \sum_{i = 1}^s \frac{\beta_i}{m_i} \log k_i \right) - \frac{\varepsilon}{3}
\\ \ge
\left( \sum_{i = 1}^s \beta_i \left( h(\mu_i) - \frac{\varepsilon}{3}
  \right) \right) - \frac{\varepsilon}{3}
=
h(\mu') - \frac{2}{3} \varepsilon
\ge
h(\mu) - \varepsilon.
\end{multline*}
Furthermore, by~\eqref{eq:1}, for each~$L\in\cL$ and~$x\in L$ we
have~$|Df^m(x)| \le e^{(\lambda(\mu) + \varepsilon)m}$, and by~\eqref{eq:2}, for each~$j\in\{1, \ldots, l \}$ we have
$$ \left| \frac{1}{m} S_m \varphi_j(x) - \int\! \varphi_j d\mu \right|
\le
\varepsilon. $$
Note that for each~$L$ in~$\cL$ we have~$|L| \ge e^{-(\lambda
  (\mu) +\varepsilon ) m}$.
  
  Let $n$ be a large integer and 
write $n=pm+q$, where $p$, $q$ are non-negative integers with $0\leq q\leq m-1$.
We have
\begin{multline*}
\frac{1}{n} \log \left| \left\{  x\in X\colon \frac{1}{n}S_n\varphi_j(x)>\alpha_j \text{ for every $j\in\{1,\ldots,l\}$} \right\} \right|
\\
\begin{aligned}
  & \geq
\frac{1}{n} \log \left| \left\{  x\in X\colon \frac{1}{pm}S_{pm}\varphi_j(x)
    > \alpha_j +\varepsilon\text{ for every $j\in\{1,\ldots,l\}$} \right\}  \right|
\\ & \ge
\frac{1}{n} \log  \left( \sum_{L \in \bigvee_{i=0}^{p-1}f^{-im}\mathcal L} \left| L \right|\right)
\\ & \ge
\frac{1}{n}\left(p\log (\# \mathcal L)   -pm( \lambda (\mu) +\varepsilon )\right)
\\ & \ge
\frac{1}{m} \log (\# \mathcal L)   -( \lambda (\mu) +2\varepsilon )
\\ &\ge
h(\mu )- \lambda(\mu) -3\varepsilon.  
\end{aligned}
\end{multline*}
Letting $n\to \infty$ and then $\varepsilon \to 0$ we obtain the desired inequality.
\end{proof}

\begin{proof}[Proof of the large deviations lower bound in the Main Theorem]
Let $f\colon X\to X$ be a map satisfying the hypotheses of
Proposition~\ref{lower-prop}, and~$\mathcal G$ an open subset of $\mathcal M$. 
Note that the topology of $\mathcal M$ has a base consisting of sets of the form 
$$\left\{\nu\in\mathcal M\colon\int\!\varphi_j d\nu> \alpha_j\text{ for every $j\in\{1,\ldots,l\}$}\right\},$$
where $l\geq1$ is an integer, each $\varphi_j\colon X\to\mathbb R$ is a continuous function
and $\alpha_j\in\mathbb R$.
Hence, there exists a collection $\{\mathcal O_\xi\}_\xi$ of sets of this form such that $\mathcal G=\bigcup_{\xi}\mathcal O_\xi$.
Proposition~\ref{lower-prop} applied to each $\mathcal O_\xi$ yields
\begin{align*}
\liminf_{n\to\infty}\frac{1}{n}\log\left|\left\{x\in X\colon\delta_{x}^n\in\mathcal G\right\}\right|
&= \liminf_{n\to\infty}\frac{1}{n}\log
\left|\left\{x\in X\colon\delta_{x}^n\in\bigcup_{\xi} {\mathcal O}_\xi \right\}\right|\\
&\geq\sup_{\xi} \liminf_{n\to\infty}\frac{1}{n}\log
\left|\left\{x\in X\colon\delta_{x}^n\in
{\mathcal O}_\xi\right\}\right| \\
&\ge \sup_{\xi} 
\sup_{ \mathcal O_\xi} F
\\ & =
\sup_{\mathcal G}F
\\ & =
-\inf_{\mathcal G}I.\qedhere
\end{align*}
\end{proof}

\section{The Uniform Scale Lemma}\label{fundamental}
This section is devoted to the proof of the following lemma that is
a key element of the proof of the large deviations upper bound in the
Main Theorem.
The large deviations upper bound is completed in Sect.\ref{s:upper bound}.

For a differentiable map~$g \colon X \to X$ and a subinterval~$J$ of~$X$ that does not contain critical points of~$g$, the \emph{distortion of~$g$ on~$J$} is by definition
$$ \sup \left\{ \frac{|Dg(x)|}{|Dg(y)|} \colon x, y \in J  \right\}. $$

\begin{USL}
Let $f\colon X\to X$ have H{\"o}lder continuous derivative and only
non-flat critical points.
Assume~$f$ is topologically exact.
Then for every $\varepsilon>0$ there exist constants $\eta_0 > 0$, $C>0$, and~$\kappa>0$, such that for every~$\eta \in (0, \eta_0)$ there is $n_0\geq1$ such that the following property holds for every integer $n\geq n_0$.
For every subinterval~$W$ of~$X$ that satisfies $\eta\leq |f^n(W)|\leq 2\eta$, 
there exists a subinterval~$J$ of~$W$ and an integer $m$ such that
$$|J|\geq e^{-\varepsilon n}|W|,\ n\leq m\leq n+C\log n,\ |f^m(J)|\geq \kappa,$$
and such that~$f^m$ maps $J$ diffeomorphically onto $f^m(J)$ with distortion
bounded by~$e^{\varepsilon n}$ (Fig. 1).
\end{USL}

In Sect.\ref{size} we establish one of the main ingredients in the proof of this
lemma, which are some general sub-exponential distortion bounds (Proposition~\ref{p:sub-exponential ratio growth}).
The first type of distortion bound is on the ratio of the sizes of two
iterated intervals, which holds for an arbitrary pull-back that is not necessarily diffeomorphic.
The second one is a sub-exponential distortion bound for diffeomorphic pull-backs with a definite ``Koebe space''.
This last distortion bound is obtained from the Koebe Principle in~\cite{dMevSt93} and a sub-exponential cross-ratio distortion bound.
In Sect.\ref{dimension} we show the abundance of ``safe points''
contained in hyperbolic sets (Lemma~\ref{l:big hyperbolic}).
This is used to apply the method of~\cite{PrzRivSmi03} to
find sub-exponentially small intervals all whose pull-backs by a high iterate
of the map are mapped diffeomorphically to unit scale.
The proof of the Uniform Scale Lemma is given in Sect.\ref{ss:proof
  of USL}.

  \begin{figure}[htb]
\begin{center}
\includegraphics[height=6cm,width=13cm]{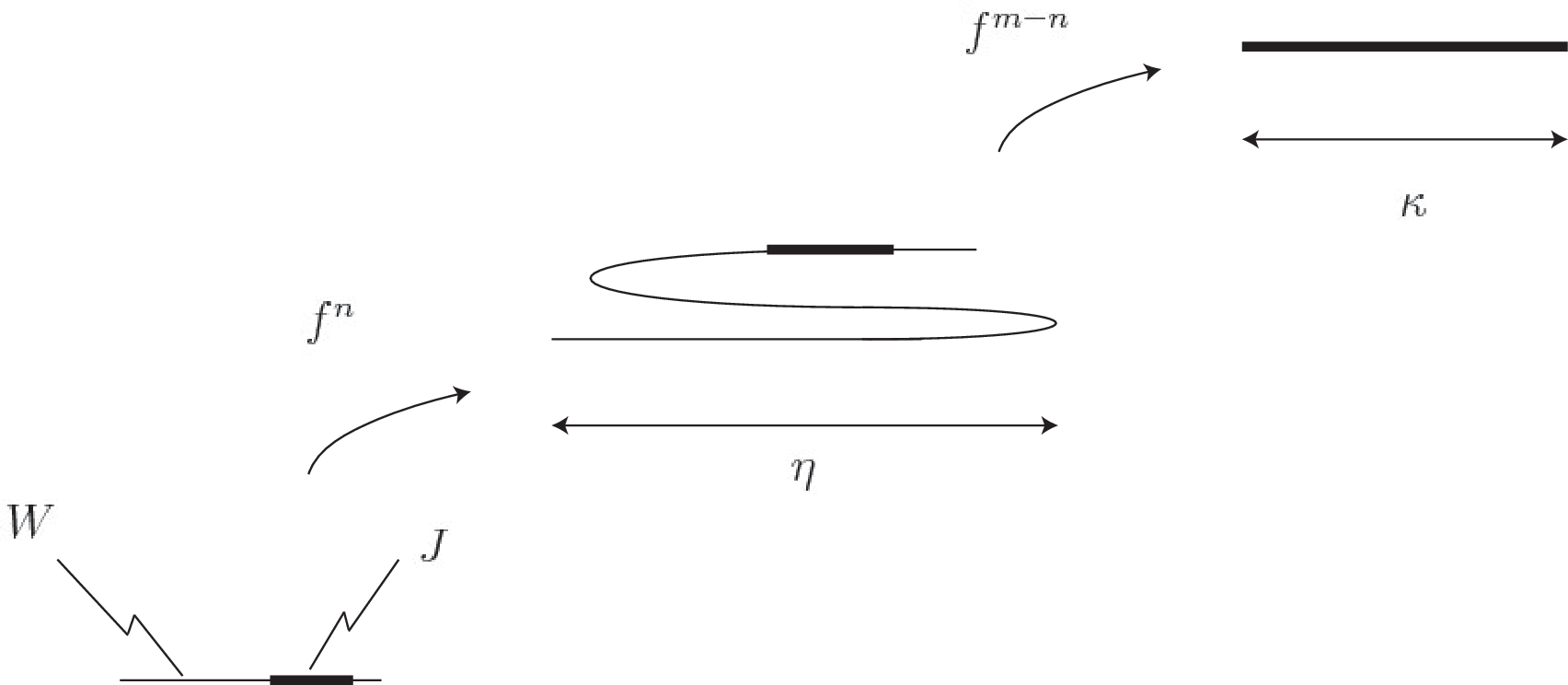}
\caption{On the Uniform Scale Lemma: for a given $\varepsilon>0$ 
one can find two small scales $\eta>0$ and $\kappa>0$ such that for every pull-back $W$ of intervals of size $\eta$
one can choose a subinterval~$J$ of~$W$ that is mapped diffeomorphically to an interval of length $\kappa$ in time $m$,
$n\leq m\leq n+C\log n$.
}
\end{center}
\end{figure}

\subsection{Sub-exponential distortion bounds}
\label{size}
In this section we prove the following proposition giving a
sub-exponential bound on the ratio of the sizes of two iterated
intervals, and a sub-exponential derivative distortion bound for
certain diffeomorphic pull-backs.
 \begin{prop}
\label{p:sub-exponential ratio growth}
Let $f\colon X\to X$ have H{\"o}lder continuous derivative and only
non-flat critical points.
Assume~$f$ is topologically exact.
Then for every $\varepsilon>0$ there exist an integer $n_1\geq1$ and $\eta_0>0$ such that 
for every integer $n\geq n_1$, every subinterval $W$ of $X$ that
satisfies $|f^n(W)|\leq 2\eta_0$, and for every subinterval $J$ of $W$,
$$\frac{|f^n(J)|}{|f^n(W)|}\leq e^{\varepsilon n}\frac{|J|}{|W|}.$$
If in addition~$f^n \colon W \to f^n(W)$ is a diffeomorphism
and~$|f^n(J)| \le \dist(\partial f^n(W), f^n(J))$, then the distortion
of~$f^n$ on~$J$ is bounded by~$e^{\varepsilon n}$.
\end{prop}

For the proof of this proposition we need the next lemma, in which we
use the assumption that each critical point is non-flat.
To state this lemma, we use the concept of ``cross-ratio'' that we proceed to
recall.
Given a subinterval~$\hJ$ of~$\R$ and an interval~$J$ whose closure is contained in the
interior of~$\hJ$, denote by~$L$ and~$R$ the connected components
of~$\hJ \setminus J$.
Then the \emph{cross-ratio $\Cr(\hJ;J)$ of~$\hJ$ and~$J$} is defined by
$$ \Cr(\hJ;J) \= \frac{|\hJ||J|}{|L||R|}. $$

\begin{lemma} \label{non-flat}
Let $f\colon X\to X$ be continuously differentiable with only non-flat critical points.
Then there exist constants $C_0>1$ and $\delta_0>0$ such that 
 for every interval $\hU$ contained in $B(\Crit(f), \delta_0)$, and every subinterval $U$ of $\hU$,  
\begin{equation*} \frac{|f(U)|}{ |f(\hU)|}
\leq C_0\frac{ |U| }{ |\hU|}. \end{equation*}
If in addition~$\hU$ is disjoint from~$\Crit(f)$ and the closure of~$U$ is contained
in the interior of~$\hU$, then
$$ \Cr(f(\hU);f(U)) \ge C_0^{-1} \Cr(\hU;U). $$
\end{lemma}

\begin{proof}
Let $c\in\Crit(f)$.
By the definition of non-flatness, 
there exist a number $\ell>1$ and diffeomorphisms~$\phi$ and~$\psi$ of~$\R$ such that  $\phi(c)=\psi(f(c))=0$
and $g=\psi\circ f\circ \phi^{-1}$ satisfies $|g(x)|=|x|^{\ell}$ for $x$ near $0$.
It is thus enough to prove the lemma with~$f$ replaced by~$g$.
For~$g$, the second inequality with~$C_0 = 1$ is given by~\cite[Property~4 in~Sect.IV.1]{dMevSt93} by noting that the Schwarzian derivative of~$g$ is
negative on~$\R \setminus \{ 0 \}$.
To prove the first inequality we treat four cases separately.
\medskip

\noindent\emph{Case 1:} $0\in U$. We have $(|U|/2)^{\ell}\leq |g(U)|\leq |U|^{\ell}$.
Since $0\in\widehat U$ we also have
$(|\widehat U|/2)^{\ell}\leq |g(\widehat U)|\leq |\widehat U|^{\ell}$.
Then $|g(U)|/|g(\widehat U)|\leq(2|U| / |\hU|)^{\ell}< 2^{\ell}|U| / |\hU|.$
\medskip

\noindent\emph{Case 2: $0\notin U$ and $0\in\widehat U$}. By the mean value theorem and the form of $g$,
there is~$\xi$ in~$U$ such that $|g(U)|=|Dg(\xi)|\cdot|U|\leq \ell|\widehat U|^{\ell-1}\cdot|U|$.
Combining this with the lower estimate of $|g(\widehat U)|$ in Case~1 yields
$|g(U)|/|g(\widehat U)|\leq 2^{\ell}\ell|U| / |\hU|.$
\medskip

\noindent\emph{Case 3: $0\notin \widehat U$ and $|\widehat U|\leq\dist(0,\widehat U)$.}
The mean value theorem gives
$|g(U)|=|Dg(\xi)|\cdot|U|$ and
$|g(\widehat U)|=|Dg(\eta)|\cdot|\widehat U|$ for some $\xi\in U$ and $\eta\in\hU$.
The assumption $|\widehat U|\leq\dist(0,\widehat U)$ implies $|\xi/\eta|\leq2$, and so
$|g(U)|/|g(\widehat U)|\leq 2^{\ell-1}|U| / |\hU|.$
\medskip

\noindent\emph{Case 4: $0\notin \widehat U$ and $|\widehat U|>\dist(0,\widehat U)$.}
Let $V$ denote the smallest closed interval containing $\widehat U$ and $0$.
We have
$|g(V)|=|g(\widehat U)|+|g(V\setminus \widehat U)|<2|g(\widehat U)|.$
Using this and the estimate in Case~2 for the pair $(U,V)$ yields
\begin{displaymath}
  |g(U)|/|g(\widehat U)|
  <
  (1/2)|g(U)|/|g(V)|
  \leq
  2^{\ell-1}\ell|U|/|V|
  <
  2^{\ell-1}\ell|U|/|\widehat U|.
  \qedhere
\end{displaymath}
\end{proof}

In the proof of Proposition~\ref{p:sub-exponential ratio growth}
we also use general properties of topologically exact maps.
First, notice that from the compactness of~$X$, for every continuous
and topologically exact map~$f \colon X \to X$ and each~$\gamma > 0$
there is an integer~$N \ge 1$ such that for every
subinterval~$J$ of~$X$ with $|J| \ge \gamma$,
we have $f^N(J)=X$; we denote by~$N(\gamma)$ the smallest such integer.

\begin{lemma}\label{upbound}
Let $f\colon X\to X$ be a continuous map that is topologically exact.
Then for every $\varepsilon>0$ there exists $\eta\in(0,1/2)$ such that for
every integer $n\geq 1$ and every subinterval~$W$ of~$X$ that satisfies $|f^n(W)|\leq\eta$,
 $|f^i(W)| \le \varepsilon$ holds for every $i\in\{0,\ldots,n-1\}$.
\end{lemma}

\begin{proof} 
Let $\eta \in (0, 1/2)$ be such that for every subinterval $V$ of $X$ that satisfies $|V|\le \eta$,
$|f^i (V)|\le 1/2$ holds for every $i\in\{0,\ldots , N(\varepsilon )-1\}$.
Let $n\ge 1$ be an integer and $W$ a subinterval of $X$ such that $|f^n(W)|\le \eta $.
If $|f^{i_0} (W)| > \varepsilon$ holds for some $i_0\in \{ 0, \ldots , n-1 \},$
then the definition of $N(\varepsilon)$ gives
$ f^{N(\varepsilon)} (f^{i_0}(W)) = X$.
Since $f(X)=X$ we get
$f^{N(\varepsilon)-1} (f^n(W)) = X,$
and this contradicts the choice of $\eta$
with $V=f^n(W)$.
\end{proof}

\begin{proof}[Proof of Proposition~\ref{p:sub-exponential ratio growth}]
In order to treat critical relations that can arise in the case $\#\Crit(f)\linebreak\geq2$ we introduce the following notion.
We say $c\in\Crit(f)$ is a \emph{tail} if $f^n(c)\notin\Crit(f)$ holds for every $n\geq1$.
Let $\Crit'(f)$ denote the set of tails.

Consider a graph made up of vertices and oriented edges between them.
The vertices are critical points of $f$. For two vertices $c_0$ and $c_1$ put an edge from
$c_0$ to $c_1$ if there exists an integer $n\geq1$
such that $f(c_0),f^2(c_0),\ldots,f^{n-1}(c_0)\notin\Crit(f)$ and $f^n(c_0)=c_1$.
The edge is labeled with $n$.
By definition, there is at most one outgoing edge from each vertex.
Since no critical point is periodic, there is no loop in the graph.
The concatenation of edges groups the set of vertices into blocks,
which might intersect.
For each block consider the sum of labels of all its edges.
Let $E$ denote the maximal sum over all blocks.
Let $\varepsilon>0$ be given and let~$C_0$ and~$\delta_0$ be the
constants given by Lemma~\ref{non-flat}. 
Choose a sufficiently large integer $n_1 \ge 1$ such
that~$e^{\varepsilon n_1 /12} \ge 2C_0^{2E}$.
Let $\delta\in(0,\delta_0)$ be such that
the set $\bigcup_{j=1}^{n_1}f^j(B(\Crit'(f), \delta))$ is disjoint from $B(\Crit(f), \delta/2)$.

Since~$f$ is continuously differentiable, there is $\kappa\in(0,\delta / 2)$ such that
for every interval~$U$ contained in~$X \setminus B(\Crit(f), \delta/2)$ that satisfies~$|U| \le \kappa$, 
\begin{equation}\label{distort}
\sup_{x,y\in U}\frac{|Df(x)|}{|Df(y)|}\leq e^{\frac{\varepsilon}{24}}.
\end{equation}
Finally, in view of Lemma~\ref{upbound} we can choose $\eta_0 > 0$
 such that for every~$\eta \in (0, \eta_0)$, every $x\in X$, every integer $n\geq 1$ and every pull-back $W$
of $B(x,\eta)$ by $f^n$, $|f^j(W)| \le \kappa$ holds for every $j\in\{0,\ldots,n-1\}$.
Note that by our choices of~$n_1$ and~$\delta$, it follows that 
\begin{equation}
  \label{eq:5}
  \#\{j\in\{0, \ldots, n - 1 \}\colon\text{$f^j(W)\cap B(\Crit(f),
    \delta/2)\neq\emptyset$}\}\leq E \left(\frac{n }{ n_1} + 1\right)
\le
\frac{2En}{n_1}.
\end{equation}

Let~$n \ge n_1$, $\eta \in (0, \eta_0)$, $W$ a pull-back of $B(x,\eta)$ by $f^n$ and~$J$ a subinterval of $W$.
For every~$j\in\{0, \ldots, n - 1 \}$ we have~$|f^j(W)| \le \kappa$.
Thus, if in addition~$f^j(W)$ is disjoint from~$B(\Crit(f), \delta/2)$, then \eqref{distort} gives
$$\frac{ |f^{j + 1}(J)| }{ |f^{j + 1}(W)|}
\le
e^{\frac{\varepsilon}{24}}\frac{ |f^j(J)| }{ |f^j(W)|}. $$
If in addition~$f^j(W)$ is disjoint from~$\Crit(f)$, then for every
subinterval~$\hU$ of~$f^j(W)$ and every interval~$U$ whose closure is contained in the
interior of~$\hU$, 
$$ \Cr(f^{j + 1}(\hU);f^{j + 1}(U))
\ge
e^{- \frac{\varepsilon}{12}} \Cr(f^j(\hU);f^j(U)). $$
Suppose now~$j\in\{0, \ldots, n - 1 \}$ is such that~$f^j(W)$ intersects~$B(\Crit(f), \delta/2)$.
Since~$\kappa \in(0, \delta/ 2)$, the interval~$f^j(W)$ is contained in~$B(\Crit(f), \delta)$, and by Lemma \ref{non-flat} we have
$$\frac{ |f^{j + 1}(J)| }{ |f^{j + 1}(W)|}
\le
C_0\frac{ |f^j(J)| }{ |f^j(W)|}. $$
If in addition~$f^j(W)$ is disjoint from~$\Crit(f)$, then for every
subinterval~$\hU$ of~$f^j(W)$ and every interval~$U$ whose closure is contained in the
interior of~$\hU$, 
$$ \Cr(f^{j + 1}(\hU);f^{j + 1}(U))
\ge
C_0^{-1} \Cr(f^j(\hU);f^j(U)). $$
Therefore, by our choice of~$n_1$ and~\eqref{eq:5} we have
$$\frac{ |f^n(J)|} { |f^n(W)|}
\le
C_0^{\frac{2En}{n_1}} e^{\frac{\varepsilon}{12} n}
\le
e^{\varepsilon n} \frac{|J| }{ |W|}, $$
which gives the first assertion of the proposition.

To prove the second assertion of the proposition, suppose~$f^n \colon W
\to f^n(W)$ is a diffeomorphism.
Then for every subinterval~$\hU$ of~$W$ and interval~$U$ whose closure is contained in the interior of~$\hU$,
$$ \frac{\Cr(f^n(\hU);f^n(U))}{\Cr(\hU;U)}
=
\prod_{j = 0}^{n - 1} \frac{\Cr(f^{j + 1}(\hU);f^{j +
    1}(U))}{\Cr(f^j(\hU);f^j(U))}
\ge
C_0^{-\frac{2En}{n_1}} e^{- \frac{\varepsilon}{12} n}
\ge
2 e^{- \frac{\varepsilon}{6} n}. $$
The Koebe Principle \cite[Theorem~IV.1.2]{dMevSt93} with~$\tau =
1$ implies that the distortion of~$f^n$ on~$J$ is bounded
by~$e^{\varepsilon n}$.
This completes the proof of the proposition.
\end{proof}

\subsection{Abundance of safe points in hyperbolic sets}\label{dimension}
Let~$f \colon X \to X$ be a differentiable interval map with at
most a finite number of critical points.
In order to carefully avoid critical points and choose diffeomorphic
pull-backs, we use the method introduced in~\cite{PrzRivSmi03}.
We adopt the terminology of ``safe points''  in \cite[Definition~12.5.7]{PrzUrb10}.
For a given $\alpha>0$ and an integer $n\geq1$ define
$$E_n(\alpha)=\bigcup_{j=1}^\infty B(f^j(\Crit(f)),\min\{n^{-\alpha },j^{-\alpha}\}).$$
Note that the set $E_n(\alpha)$ is decreasing in $n$. Set
$$E(\alpha)=\bigcap_{n=1}^\infty E_n(\alpha).$$
Note that $E(\alpha)$ contains $\bigcup_{j=1}^\infty f^j(\Crit(f))$.

We say $x\in X$ is $\alpha$-\emph{safe} if $x\notin E(\alpha)$.
If $x$ is $\alpha$-\emph{safe}, then for every integer
 $n\geq 1$ with $x\notin E_n(\alpha)$ the ball $B(x,n^{-\alpha})$ is disjoint from
$\bigcup_{j=1}^n f^j(\Crit(f)).$ Hence,
the pull-backs of $B(x,n^{-\alpha})$ by $f^n$ are diffeomorphic.

\begin{lemma}\label{HD}
For every $\alpha>0$,
$\HD(E(\alpha))\leq \alpha^{-1}$.
\end{lemma}
\begin{proof}
 For each $n$ consider the covering of $E(\alpha)$
by the intervals $$B(f^j(c),\min\{n^{-\alpha},j^{-\alpha}\}),
\ c\in\Crit(f),\ j\in\{1,2,\ldots\}.$$
Let $\beta>\alpha^{-1}$. We have
\begin{align*}
\sum_{c\in\Crit(f)}\sum_{j=1}^\infty |B(f^j(c),\min\{n^{-\alpha},j^{-\alpha}\})|^\beta&=
\sum_{c\in\Crit(f)}\left(\sum_{j=1}^n 
+\sum_{j=n+1}^\infty \right)\\
&\leq \#\Crit(f)\cdot\left(2^\beta n^{1-\alpha\beta}+\sum_{j=n+1}^\infty 2^{\beta}j^{-\alpha\beta}\right).
\end{align*}
This number goes to $0$ as $n\to\infty$, and so
the Hausdorff $\beta$-measure of $E(\alpha)$ is $0$.
Since $\beta>\alpha^{-1}$ is arbitrary we obtain $\HD(E(\alpha))\leq\alpha^{-1}$.
\end{proof}

\begin{lemma}
\label{l:big hyperbolic}
Let $f \colon X \to X$ have H{\"o}lder continuous derivative and at most
a finite number of critical points.
If~$f$ is topologically exact, then there is~$\alpha > 0$ such that the following property holds.
For every~$\eta > 0$ there is a hyperbolic set~$\Lambda$ of~$f$ such that
for every~$x \in X$, the set $B(x,\eta)\cap \Lambda$ is nonempty and contains an $\alpha$-safe point.
\end{lemma}
\begin{proof}
Since $f$ is topologically exact, there exist an integer $n>0$ and a
closed subset $\widehat A$ of $X$ such that $f^n(\widehat A)\subset \widehat A$ and
 $f^n\colon\widehat A\to f^n(\widehat A)$ is topologically conjugate to the one-sided full shift on two symbols.
Hence, $f$ has positive topological entropy, see also~\cite[Proposition~4.70]{Rue17}.
From the variational principle, see for
example~\cite[Theorem~4.4.11]{Kel98} or~\cite[Theorem~3.4.1]{PrzUrb10}, there is a measure~$\mu$ in~$\cM(f)$ satisfying~$h(\mu) >
0$, and therefore~$\lambda(\mu) > 0$ by Ruelle's inequality.
By Lemma~\ref{katok} with $\varepsilon = \lambda(\mu)/2$, there are integers~$k \ge 2$ and~$m \ge 1$, a closed subinterval~$K$ of~$X$
and pairwise disjoint closed subintervals~$K_1$, \ldots, $K_k$ of~$K$, such that for each~$i$ in~$\{1, \ldots, k \}$ the map~$f^m \colon K_i \to K$ is a diffeomorphism and~$|Df^m| \ge \exp(\lambda(\mu)m/2)$ on~$K_i$.
It follows that the maximal invariant set~$\hLambda_0$ of~$f^m$
on~$\bigcup_{i = 1}^k K_i$ is a hyperbolic set for~$f^m$.
Since~$k \ge 2$, we have~$\HD(\hLambda_0) > 0$.

Let~$Q \ge 2 \eta^{-1}$ be an integer and put~$\xi \= \exp(\lambda(\mu) m/2)$.
Since~$f$ is topologically exact, the map~$f^m$ is also topologically exact, so there is an integer~$N \ge 1$ such that 
$f^{Nm} \left( \left( \frac{i - 1}{Q}, \frac{i}{Q} \right) \right) = X$ holds
for each $i\in\{1, \ldots, Q \}$. 
Let~$p_0$ be a point in the uncountable set~$\hLambda_0$ that is not in~$\bigcup_{j = 1}^{\infty} f^j(\Crit(f))$.
Define recursively for each~$i\in\{1, \ldots, Q \}$ a point~$p_i\in\left( \frac{i - 1}{Q}, \frac{i}{Q} \right)$, so that~$f^{Nm}(p_i) = p_{i - 1}$.
Using again that~$f^m$ is topologically exact, we can find an integer~$N' \ge 1$ and a point~$p$ in the interior of~$K$ that is not in~$\hLambda_0$, such that~$f^{N' m}(p) = p_Q$.
Defining~$\ell \= QN + N'$, we have that~$f^{\ell m}(p) = p_0$ and that the set
$$ \{ p, f^m(p), \ldots, f^{\ell m}(p) \}
\supset
\{ p_1, p_2, \ldots, p_Q \} $$
is $\eta$-dense in~$X$.
Since~$p_0$ is not in~$\bigcup_{j = 1}^{\infty} f^j(\Crit(f))$, there is~$\delta_0 > 0$ such that~$B(p_0, \delta_0)$ is disjoint from~$\bigcup_{j = 1}^{\ell m} f^j(\Crit(f))$.
It follows that the pull-back~$W_0$ of~$B(p_0, \delta_0)$ by~$f^{\ell m}$ containing~$p$ is diffeomorphic.
Reduce~$\delta_0$ if necessary so that~$W_0$ is contained in~$K$.
Let~$\ell_0 \ge 1$ be a sufficiently large integer such that $\xi^{-\ell_0} < \inf_{W_0} |Df^{\ell m}|$ and such that the pull-back of~$K$ by~$f^{\ell_0 m}$ containing~$p_0$ is contained
in~$B(p_0, \delta_0)$.
Since~$p_0$ is in~$\hLambda_0$, it follows that this last pull-back is diffeomorphic.
We conclude that, if we put~$M \= (\ell + \ell_0)m$, then the pull-back~$L_0$ of~$K$ by~$f^M$ containing~$p$ is diffeomorphic.
Moreover, from our choice of~$\ell_0$ we have
\begin{equation}
  \label{eq:3}
\inf_{L_0} |Df^M|
\ge
\xi^{\ell_0} \inf_{W_0}|Df^{\ell m}|
>
1.  
\end{equation}
Let~$\cL$ be the collection formed by~$L_0$ and by all pull-backs of~$K$ by~$f^M$ that
intersect~$\hLambda_0$.
Since ~$\inf_{K_i}|Df^m| \ge \xi$ for each~$i\in\{1, \ldots, k \}$,
$\inf_L |Df^M| \ge \xi^{\ell + \ell_0} > 1$ holds for every $L\in\cL$ different from~$L_0$.
Together with~\eqref{eq:3} this implies that the maximal invariant set~$\hLambda$ of~$f^M$ in~$\bigcup_{L \in \cL} L$ is a hyperbolic set for~$f^M$, and that~$f^M \colon \hLambda \to \hLambda$ is topologically exact.
On the other hand, the point~$p$ is by definition in~$L_0$ and~$f^M(p) = f^{\ell_0 m}(p_0)$ is in~$\hLambda_0$.
This implies $p\in\hLambda$ and therefore~$\hLambda$ is $\eta$-dense on~$X$.
So, for every~$x\in X$ the ball~$B(x, \eta)$ intersects~$\hLambda$ and, since~$f^M \colon \hLambda \to \hLambda$ is topologically exact, it follows that there is an integer~$k \ge 1$ such that~$f^{kM}(B(x, \eta) \cap \hLambda) = \hLambda$.
Using that~$f^{k M}$ is Lipschitz continuous on~$\hLambda$ and that~$\hLambda$ contains~$\hLambda_0$, we obtain
$$ \HD(B(x, \eta) \cap \hLambda)
\ge
\HD(\hLambda)
\ge
\HD(\hLambda_0). $$
In view of Lemma~\ref{HD}, this proves the lemma with~$\alpha = \frac{2}{\HD(\hLambda_0)}$ and with the hyperbolic set for~$f$ defined by~$\Lambda \= \bigcup_{i = 0}^{M - 1} f^i(\hLambda)$.
\end{proof}

\subsection{Proof of the Uniform Scale Lemma}
\label{ss:proof of USL}
Let $\varepsilon>0$ be given.
Let $n_1$ and $\eta_0 > 0$ be such that the conclusions of
Proposition~\ref{p:sub-exponential ratio growth} hold with~$\varepsilon$ replaced by~$\varepsilon/2$.
Fix $\eta \in (0, \eta_0)$, and let~$\alpha$ and~$\Lambda$ be given by Lemma~\ref{l:big hyperbolic} with~$\eta$ replaced by $\eta/6$.
Since~$\Lambda$ is a hyperbolic set for~$f$, there exist constants
$C_0>0$, $\kappa > 0$, $\lambda > 1$ such that for every~$x \in X$ and every integer $n\geq1$ such that $\dist(f^{i}(x),\Lambda)\leq 3\kappa$ for every  $i\in\{0,1,\ldots, n-1\}$, 
 $|Df^n(x)|\geq C_0 \lambda^n$ holds.
It follows that there is a constant~$C_1 > 0$ such that for every interval~$U$ intersecting~$\Lambda$ and satisfying $|U| \le 3\kappa$,
there is an integer $k\geq0$ such that
\begin{equation}
\label{choicekappa}
 k\leq C_1\log(1/|U|), \quad
3 \kappa \le |f^k(U)|\le 3 \kappa \cdot \sup_X |Df|, 
\end{equation}
and such that~$f^k$ maps~$U$ diffeomorphically onto~$f^k(U)$.
Reduce~$\kappa$ if necessary, so that~$\kappa \le \eta / (3 \sup_X|Df|)$,
and so that for every~$U$ and~$k$ as above we have in addition that
the distortion of~$f^k$ on~$U$ is bounded by~$2$.

By Lemma \ref{l:big hyperbolic},
 each ball of radius $\eta/6$ contains an $\alpha$-safe point in $\Lambda$.
From this and the compactness of $X$,
we can find a finite subset~$F$ of $\Lambda \setminus E(\alpha)$ that
is $(\eta/3)$-dense in~$X$.
Let $n_0\geq n_1$ be a sufficiently large integer so that~$F$ is
disjoint from~$E_{n_0}(\alpha)$,
\begin{equation}
\label{expansion1}
n_0^{-\alpha}\le \min \left\{ \frac{\eta}{6}, \frac{3}{2} \kappa \right\}
\text{ and }
\frac{n_0^{-\alpha}}{12\eta}\geq e^{-\frac{\varepsilon}{2} n_0}.
\end{equation}

Now, let $n\geq n_0$ be an integer, and~$W$ a subinterval
of~$X$ that satisfies $\eta\leq|f^n(W)|\leq 2\eta$.
Since the finite set $F$ is $(\eta/3)$-dense, there is a point $x \in F$ whose 
distance to the mid point of~$f^n(W)$ is at most $\eta/3$.
Since $|f^n(W)|\geq \eta$ it follows that $B(x,\eta/6)$ is contained in~$f^n(W)$.
Together with the first inequality in \eqref{expansion1} this implies
that $U = B(x,n^{-\alpha})$ is contained in~$f^n(W)$.
Since by construction~$x\notin E_{n_0}(\alpha)$, every pull-back of~$U$ by~$f^n$ is diffeomorphic.
Take one pull-back of~$U$ by~$f^n$ contained in~$W$ and denote it by~$\hJ$.

Since $x \in \Lambda$ and~$|U| = |B(x,n^{-\alpha})| \le 3 \kappa$ by the first inequality
in~\eqref{expansion1}, there is an integer $k\geq0$ such that
$$ k\leq C_1\log(1/|U|)\leq C_1\alpha\log n, \quad
3\kappa \le |f^k(U)| \le \eta$$
by \eqref{choicekappa},
and such that~$f^k$ maps~$U$ diffeomorphically onto~$f^k(U)$ with distortion
bounded by~2.
So, if we put $m=n+k$, then $n\leq m\leq n + C_1 \alpha \log n$ and $f^m$ maps~$\hJ$
diffeomorphically onto $f^m(\hJ)$.
Denote by~$J \subset W$ the pull-back by~$f^m$ of the interval with
the same center as~$f^m(\hJ)$ and whose length is equal to~$\frac{1}{3} |f^m(\hJ)|$.
By Proposition~\ref{p:sub-exponential ratio growth} with~$n = m$
and~$W = \hJ$, the distortion of~$f^m$ on~$J$ is bounded
by~$e^{\varepsilon n}$.
Note furthermore that
$$ |f^m(J)|
=
\frac{1}{3} |f^m(\hJ)|
=
\frac{1}{3} |f^k(U)|
\ge
\kappa. $$
On the other hand, by Proposition~\ref{p:sub-exponential
  ratio growth} and the fact that the distortion of~$f^k$ on~$U =
f^n(\hJ)$ is bounded by~$2$, we have
$$ \frac{n^{-\alpha}}{12\eta}
\leq
\frac{1}{6} \cdot \frac{ |U|}{|f^n(W)|}
\le
\frac{|f^n(J)|}{|f^n(W)|}
\leq
e^{\frac{\varepsilon}{2} n}\frac{|J|}{|W|}. $$
By the second inequality in \eqref{expansion1} this implies $|J|\geq
e^{-\varepsilon n}|W|$, and completes the proof of the lemma with~$C = \alpha C_1$.

\section{The large deviations upper bound}
\label{s:upper bound}
In this section we complete the proof of the large deviations upper
bound in the Main Theorem.
In Sect.\ref{covering} we construct
certain horseshoes (Proposition~\ref{horseshoe}) that are tailored to
a given basic open set of~$\cM(f)$.
The construction is based on the Uniform Scale Lemma in Sect.\ref{fundamental}.
In order to treat inflection critical points, initially we restrict ourselves to small intervals.
In Sect.\ref{intermediate} we prove two intermediate estimates.
The first is restricted to a small interval
(Proposition~\ref{upper0}), and the second is a global estimate
(Proposition~\ref{upper}) obtained by spreading out the local estimate.
In Sect.\ref{end} we complete the large deviations upper bound.

Positive constants we will be concerned with for the rest of this paper are 
$\varepsilon$, $\eta$,
$\kappa$, $\rho$, chosen in this order.
The purposes of them are as follows:

\begin{itemize}

\item $\varepsilon$ is the error tolerance in the statement of Proposition \ref{upper};

\item $\kappa$ determines the scale of intervals given by the Uniform Scale Lemma;

\item $\eta$ determines the scale of the images of pull-backs of intervals;

\item $\rho$ determines the scale of horseshoes (see Proposition \ref{horseshoe}).

\end{itemize}

\subsection{Horseshoe argument}\label{covering}
Let $f\colon X\to X$ be a topologically exact continuous map.
Let $n\geq1$ be an integer and $\eta$ in~$(0, 1/2)$.
Put~$M \= [1/\eta] + 1$ and note that~$1/M < \eta < 3/(2M)$.
 Set $x_k=k/M$ for each $k\in\{1,2,\ldots, M - 1 \}$, and let
 $\mathcal W_n(x_k,\eta)$ denote the collection of all pull-backs $W$ of $B(x_k,\eta)$ by $f^n$
 that satisfy $x_k\in f^n(W)$. 
 Note that elements of $\mathcal W_n(x_k,\eta)$ are pairwise disjoint.
We now define $$\mathcal P_n(\eta)=\bigcup_{k=1}^{M - 1}\mathcal W_n(x_k,\eta).$$
It is easy to see that $\mathcal P_n(\eta)$ has the following properties:

\begin{itemize}

\item for every $x\in X$ there exists 
 $W\in\mathcal P_{n}(\eta)$
such that $x\in W$;

\item for every $W\in\mathcal P_n(\eta)$, we have $\eta \le |f^n(W)|
  \le 2\eta$;

\item every element of $ \mathcal P_{n}(\eta)$ can intersect at most two others on the boundary and two others in the interior.
If $W_1, W_2\in\mathcal P_n(\eta)$ and $\interior(W_1)\cap \interior(W_2)\neq\emptyset$, then 
for some $k\in\{2,\ldots,M - 1\}$,
$$ \{W_1,W_2\}
\subset
\cW_n(x_{k-1},\eta) \cup \cW_n(x_k,\eta) \cup \cW_n(x_{k+1},\eta).$$
\end{itemize}

The first two items follow from $f(X)=X$.
The last one is immediate from the definitions, see FIGURE 2.

 \begin{figure}[htb]
 \begin{center}
 \includegraphics[height=4.5cm,width=5cm]{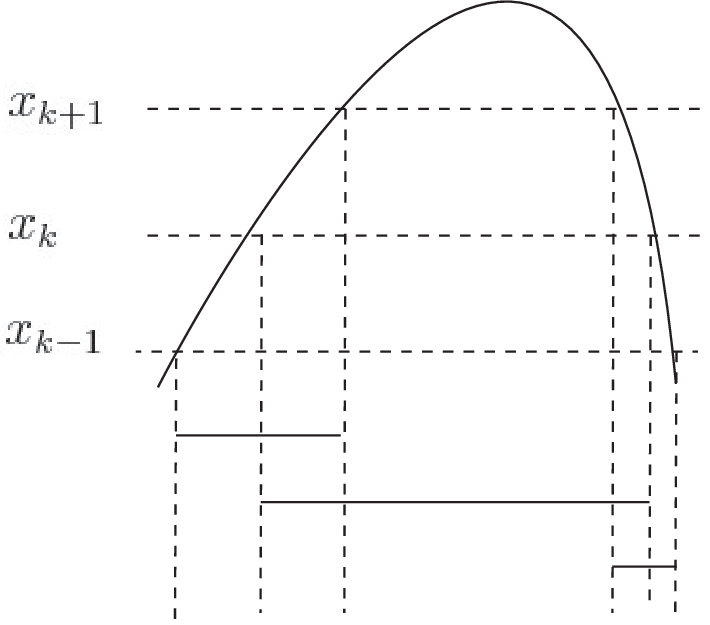}
 \caption{part of the graph of $f^n$ and the partition of $\mathcal P_n(\eta)$.
 Every element of $\mathcal P_n(\eta)$ intersects no more than two other elements in their interiors.}
 \end{center}
 \end{figure}

Fix once and for all a point 
$x_0\in\interior{X}$ such that $x_0\notin \bigcup_{n=1}^\infty f^n(\Crit(f))$. 

\begin{prop}\label{horseshoe}
Let $f\colon X\to X$ have H{\"o}lder continuous derivative and only
non-flat critical points.
Assume~$f$ is topologically exact.
Then for every $\varepsilon_0>0$ there exist $\eta>0$, $C>0$ and $\rho > 0$ such that
$B(x_0,2\rho)\cap\partial X=\emptyset$, and
 the following holds. 
Let $l\geq1$ be an integer, 
$\varphi_1,\ldots,\varphi_l\colon X\to\mathbb R$ be 
continuous functions 
and let $\alpha_1,\ldots,\alpha_l\in\mathbb R$.
For each integer~$n$ define
\begin{equation}
  \label{eq:basic open}
H_n
\=
\left\{ x \in X \colon \text{for every~$j$ in~$\{1, \ldots, l\}$ we
    have~$\frac{1}{n} S_n\varphi_j (x) \ge \alpha_j$} \right\}  
\end{equation}
and
$$ \mathcal Q_{n}
=
\left\{W \in \mathcal P_{n}(\eta) \text{ intersecting~$H_n \cap B(x_0,
  \rho)$}\right\}.$$
Then, for each sufficiently large integer $n\geq 1$ such that~$\cQ_n$ is nonempty,
  there exist an integer $q\geq n$ and
pairwise disjoint diffeomorphic pull-backs
  $L_1,\ldots,L_t$ of $B(x_0,2\rho)$ by $f^q$ contained in $B(x_0,2\rho)$ such that:

\begin{itemize}
\item[(a)] $n\leq q\leq n+C\log n$;

\item[(b)] for each~$i$ in~$\{1, \ldots, t\}$ the distortion of~$f^q$
  on~$L_i$ is bounded by~$e^{\varepsilon_0 n}$, the interval~$L_i$ is contained in some $W\in\mathcal Q_n$, and $\sum_{W\in\mathcal Q_n}|W|\leq   e^{\varepsilon_0 n}\sum_{i=1}^t|L_i|$;

\item[(c)] 
for every $x\in \bigcup_{i=1}^t L_i$ and $j\in\{1,\ldots,l\}$, we have
$\frac{1}{q}S_q\varphi_j(x) > \alpha_j - \varepsilon_0$.
\end{itemize}
\end{prop}

\begin{proof}
Let $\varepsilon_0>0$. Since each $\varphi_j$ $(j=1,\ldots,l)$ is uniformly continuous,
there exists $\varepsilon\in(0,\varepsilon_0)$ such that
if $x,y\in X$ and $|x-y|\leq\varepsilon$ then $|\varphi_j(x)-\varphi_j(y)|\leq \varepsilon_0/2$.
Let $\eta_0$, $C$ and $\kappa$ be the constants for which 
the conclusion of the Uniform Scale Lemma holds with~$\varepsilon$ replaced by~$\varepsilon/4$.
Fix~$\eta \in (0, \eta_0)$ sufficiently small so that for every subinterval~$W$ of~$X$ and every integer~$m \ge 1$ such that~$|f^m(W)| \le 2\eta$, we have for each~$j \in \{0, \ldots, m - 1 \}$ the estimate~$|f^j(W)| \le \varepsilon$ (Lemma~\ref{upbound}).
 Recall that $N(\kappa)\geq1$ is the smallest integer such that 
 for every subinterval~$J$ of~$X$ with $|J|\geq\kappa$, $f^{N(\kappa)}(J)=X$, see Sect.\ref{fundamental}.
Let $\rho_0>0$ be sufficiently small
such that~$B(x_0,2\rho_0)\cap\partial X=\emptyset$
and $B(x_0,2\rho_0)$ is disjoint from~$\bigcup_{i=1}^{N(\kappa)}f^i(\Crit(f))=\emptyset$.
The last condition is indeed realized by our assumption $x_0 \notin \bigcup_{n=1}^\infty f^n(\Crit(f))$,
and it implies that each pull-back of~$B(x_0, 2\rho_0)$ by~$f^{N(\kappa)}$ is diffeomorphic.
Let~$\rho\in(0, \min\{\rho_0,\kappa\})$ be sufficiently small so that the distortion of~$f^{N(\kappa)}$ on each pull-back of~$B(x_0, 2\rho)$ by~$f^{N(\kappa)}$ 
is bounded by~$2$.

\begin{lemma}\label{sublem1}
For every integer
$n\geq N(\rho)$ and every $W\in\mathcal P_n(\eta)$
intersecting~$B(x_0, \rho)$,
we have $W\subset B(x_0,2\rho)$.
\end{lemma}
\begin{proof} From the definition of $N(\rho)$ in Sect.\ref{fundamental}, 
for every integer $n\geq N(\rho)$ and every pull-back $W\in\mathcal P_n(\eta)$, we have $|W|\leq\rho$.
So $W\cap B(x_0,\rho)\neq\emptyset$ implies~$W \subset B(x_0,2\rho)$.
\end{proof}

Let $n\geq \max\{n_0,N(\rho)\}$.
By the Uniform Scale Lemma it is possible to choose for each 
$W\in\mathcal Q_n$ a closed subinterval $J_W\subset W$ and
an integer $m_W\geq1$ such that the following holds:
$$|J_W|\geq e^{-\frac{\varepsilon}{4} n}|W|,
\ n\leq m_W\leq n+C\log n,
\ |f^{m_W}(J_W)|\geq \kappa,$$
and~$f^{m_W}$ maps $J_W$ diffeomorphically onto $f^{m_W}(J_W)$ with
distortion bounded by~$e^{\frac{\varepsilon}{4} n}$.
Set
$$ \mathcal Q_n(p)
=
\{W\in\mathcal Q_n\colon m_W=p\}. $$
Let~$p_0$ denote a value of~$p$ that maximizes
$\sum_{W\in\mathcal Q_n(p)}|W|$, so
\begin{equation}\label{horseshoe3}
\sum_{W\in\mathcal Q_n(p_0)}|W|\geq\frac{1}{1 + C\log n}\sum_{W\in\mathcal Q_n}|W|.\end{equation}

Set $q=p_0+N(\kappa)$, and note that for every sufficiently large~$n$
item~(a) holds with~$C$ replaced by~$2C$.
Since for each $W\in\mathcal Q_n(p_0)$ 
we have $|f^{p_0}(J_W)|\geq\kappa$, $J_W$ contains at least one pull-back of~$B(x_0, 2\rho)$ by~$f^q$.
Moreover, since the map $f^{p_0}\colon J_W\to
f^{p_0}(J_W)$ is diffeomorphic,
 every pull-back of $B(x_0,2\rho)$ by~$f^{q}$ that is contained in 
 $J_W$ is diffeomorphic. 
Pick one of these diffeomorphic pull-backs and denote it by
$L_W$.
Since by the Uniform Scale Lemma the distortion of~$f^{p_0}$ on~$J_W$
is bounded by~$e^{\frac{\varepsilon}{4} n}$, and since by our choice
of~$\rho$ the distortion of~$f^{N(\kappa)} = f^{q - p_0}$ on~$f^{p_0}(L_W)$ is
bounded by~$2$, it follows that the distortion of~$f^q$ on~$L_W$ is
bounded by~$e^{\varepsilon n}$, provided that~$n$ is sufficiently large.

\begin{lemma}\label{horseshoe2}
For every sufficiently large~$n$ and $W\in\mathcal Q_n(p_0)$, we have
$ |L_W|\geq e^{- \frac{3}{4} \varepsilon n}|W|. $
\end{lemma}
\begin{proof}
Since $|f^{q}(L_W)|=4\rho$ and $q-n \leq 2C \log n$, we have
$$ |f^{n}(L_W)|
\geq
|f^{q}(L_W)| \left(\sup_X |Df|\right)^{-(q-n)}
\geq
4\rho\left(\sup_X |Df|\right)^{-2C \log n}. $$
Using~$|f^n(J_W)| \le |f^n(W)| \leq 2\eta$, and that the distortion of~$f^{p_0}$
on~$J_W$ is bounded by~$e^{\frac{\varepsilon}{4} n}$, we also have
$$\frac{|L_W|}{|J_W|}
\geq
e^{- \frac{\varepsilon}{4} n} \frac{|f^{n}(L_W)|}{|f^{n}(J_W)|}
\geq
e^{- \frac{\varepsilon}{4} n} \frac{4\rho\left(\sup_X |Df|\right)^{- 2 C\log n}}{2\eta}. $$
Together with the inequality~$|J_W| \ge e^{-\frac{\varepsilon}{4} n}
|W|$, this completes the proof.
\end{proof}

Any two elements of the collection of intervals
$\{L_W\colon W\in\mathcal Q_n(p_0)\}$
are either disjoint or coincide with each other.
Moreover, each of these intervals intersects at most five elements of~$\{ L_W \colon W \in \mathcal Q_n \}$.
Let $\{L_i\}_{i=1}^t$ denote a collection of distinct elements of
$\{L_W\colon W\in\mathcal Q_n(p_0)\}$ that maximizes~$\sum_{i = 1}^t |L_i|$.
Using~\eqref{horseshoe3} and Lemma~\ref{horseshoe2}, for every large integer $n\geq1$ we have
\begin{displaymath}
\sum_{i=1}^t|L_i|
 \geq
\frac{1}{5}\sum_{W\in\mathcal Q_n(p_0)}|L_W|
\geq \frac{1}{5}e^{-\frac{3}{4} \varepsilon n}\sum_{W\in\mathcal Q_n(p_0)}|W|\geq e^{-\varepsilon n}\sum_{W\in\mathcal Q_n}|W|.
\end{displaymath}
By Lemma \ref{sublem1}, $L_i\subset B(x_0,2\rho)$. 
Since $\varepsilon\in(0,\varepsilon_0)$
this completes the proof of item (b).

It is left to prove item (c). 
Since $L_W\subset J_W\subset W$
for every $W\in\mathcal Q_n(p_0)$,
it suffices to prove the inequality for every $x\in \bigcup_{W\in\mathcal Q_n} W$.
To ease notation, write $\varphi,\alpha$ for $\varphi_j,\alpha_j$
respectively.
Let $W\in\mathcal Q_n$, choose a point $x\in W$ such that $S_n\varphi(x)\geq\alpha n$,
and let $y\in W$.
By our choice of~$\eta$ we have $|f^i(L_W)|\leq|f^i(W)|\leq\varepsilon$ for every
$i\in\{0,\ldots, n-1\}$, so
\begin{align*}
\frac{1}{n} |S_n\varphi(x)-S_n\varphi(y)|&\leq \frac{\varepsilon_0}{2}.\end{align*}
Since
$$ S_q\varphi(y)=S_n\varphi(y)+S_{q-n}\varphi(f^ny)\geq
S_n\varphi(x)-|S_n\varphi(x)-S_n\varphi(y)|-(q-n) \sup_X |\varphi| $$
and $0\leq q-n\leq 2C \log n$, for large $n$ we have
\begin{align*}\frac{1}{q}S_q\varphi(y)
&\geq 
\frac{1}{q}S_n\varphi(x)- \frac{n\varepsilon_0}{2q}-\frac{q-n}{q} \sup_X |\varphi|
\\ &\geq
\frac{n}{q}\alpha - \frac{\varepsilon_0}{2}
- \frac{2 C \log n}{n} \sup_X |\varphi|
\\ & >
\alpha - \varepsilon_0 .\end{align*}
This completes the proof of item~(c) and of the proposition.
 \end{proof}

\subsection{Intermediate estimates}\label{intermediate}
Using Proposition \ref{horseshoe} we prove two propositions.
The first one (Proposition~\ref{upper0}) is a local estimate near the point~$x_0$ chosen before
Proposition~\ref{horseshoe}.
The second proposition (Proposition~\ref{upper}) is a global estimate
that is obtained by using the topological exactness of~$f$ to spread
out the local estimate.
\begin{prop}
\label{upper0}
Let $f\colon X\to X$ have H{\"o}lder continuous derivative and only
non-flat critical points.
Assume~$f$ is topologically exact.
Then for every $\varepsilon_0>0$ there exists $\rho>0$ such that
 the following holds. 
  Let $l\geq1$ be an integer,
$\varphi_1,\ldots,\varphi_l\colon X\to\mathbb R$ be 
continuous functions and let $\alpha_1,\ldots,\alpha_l\in\mathbb R$.
Then there exists an integer $n_{0}\geq1$ such that, if
$n\geq n_{0}$ is an integer for which the set~$H_n$ defined
by~\eqref{eq:basic open} is non-empty, then there exists $\mu\in\mathcal M(f)$ such that 
$$ \int\! \varphi_jd\mu
>
\alpha_j -  \varepsilon_0 \ \text{ for every $j\in\{1,\ldots,l\}$},$$
and
$$ \frac{1}{n} \log \left| H_n \cap B(x_0,\rho) \right|
\leq
F(\mu)+ \varepsilon_0. $$
\end{prop}
The proof of this proposition is given after the following lemma.
The next lemma will be proved along the standard line of the ergodic theory of uniformly hyperbolic systems.

\begin{lemma}\label{horse}
Let~$f \colon X \to X$ have continuous derivative and at most a finite
number of critical points.
Moreover, let~$B$ be a subinterval of~$X$, $t, q \ge 1$ integers, and let $L_1,\ldots,L_t$ be pairwise disjoint diffeomorphic
pull-backs of~$B$ by~$f^q$ contained in~$B$.
Finally, let~$\Delta > 1$ be a constant such that for each~$i$ in~$\{1, \ldots, t \}$ the distortion of~$f^q$
on~$L_i$ is bounded by~$\Delta$.
Then there exists $\hmu\in\mathcal M(f^q)$ supported on~$L_1 \cup
\cdots \cup L_t$, such that the measure~$\mu \= \frac{1}{q} (\hmu +
\cdots + f_*^{q - 1} \hmu)$ in~$\cM(f)$ satisfies
$$ qF(\mu)
\geq
\log \left( \frac{|L_1| + \cdots + |L_t|}{\Delta |B|} \right).$$
\end{lemma}

Recall that for a continuous map~$f \colon X \to X$, an integer~$n \ge 1$ and~$\varepsilon > 0$, a subset~$Y$ of~$X$ is \emph{$(n,\varepsilon)$-separated} if for each distinct~$y$ and~$y'$ in~$Y$ there is~$j$ in~$\{0, \ldots, n - 1 \}$ such that~$|f^j(y) - f^j(y')| \ge \varepsilon$. 
\begin{proof}
Let~$K$ be the maximal invariant set of~$f^q$ on~$L_1 \cup \cdots
\cup L_t$, and fix a point~$y_0$ in this set.
Moreover, put
$$ \varepsilon
\= 
\min \{ \dist(L_i, L_j) \colon i, j \in \{ 1, \ldots, t
\} \text{ distinct} \}, $$
and note that for every integer~$n \ge 1$ the set~$(f^q|_K)^{-n}(y_0)$ is
  $(n, \varepsilon)$-separated for~$f^q|_K$.
From the definition of topological pressure in terms of~$(n, \varepsilon)$-separated sets and the variational
principle, this implies
\begin{displaymath}
  \sup_{\hnu \in \cM(f^q|_K)} \left( h_{f^q|_K}(\hnu)
- \int \log |Df^q| d \hnu \right)
\ge
\limsup_{n \to \infty} \frac{1}{n} \log \left( \sum_{x \in \left(
      f^q|_K \right)^{-n}(y_0)} |Df^{qn}(x)|^{-1} \right),
\end{displaymath}
where $\cM(f^q|_K)$ denotes the set of $f^q|_K$-invariant Borel probability measures
and  $h_{f^q|_K}(\hnu)$ denotes the entropy of $\hnu\in\cM(f^q|_K)$.
See for example~\cite[Theorem~4.4.11]{Kel98} or~\cite[Theorems~3.3.2 and~3.4.1]{PrzUrb10}.
Using that for each~$i$ in~$\{1, \ldots, t\}$ the distortion of~$f^q$
on~$L_i$ is bounded by~$\Delta$, we have for every~$n \ge 1$
\begin{equation*}
    \sum_{x \in \left( f^q|_K \right)^{-n}(y_0)} |Df^{qn}(x)|^{-1}
 \ge
\left( \inf_{y' \in K} \sum_{x' \in \left( f^q|_K \right)^{-1}(y')} |Df^q(x')|^{-1}
\right)^n
 \ge
\left( \frac{|L_1| + \cdots + |L_t|}{\Delta |B|} \right)^n.
\end{equation*}
We thus obtain
$$   \sup_{\hnu \in \cM(f^q|_K)} \left( h_{f^q|_K}(\hnu)
- \int \log |Df^q| d \hnu \right)
\ge
\log \left( \frac{|L_1| + \cdots + |L_t|}{\Delta |B|} \right). $$
Since the measure-theoretic entropy of~$f^q$ is
upper semi-continuous~\cite[Corollary~2]{MisSzl80}, the supremum above is
attained.
Then the lemma follows from the fact that for
each~$\hnu$ in~$\cM(f^q|_K)$, the measure~$\nu \= \frac{1}{q} (\hnu +
f_* \hnu + \cdots + f_*^{q - 1} \hnu)$ is in~$\cM(f)$ and satisfies
\begin{equation*} h_{f^q|_K}(\hnu) - \int \log |Df^q| d \hnu
= q F(\nu).\qedhere\end{equation*}
\end{proof}

\begin{proof}[Proof of Proposition~\ref{upper0}]
Let $\varepsilon_0>0$.
Take constants $\eta$, $C$, $\rho$, 
a positive integer $q$, and a collection of pairwise disjoint closed intervals $L_1,\ldots,L_t$  for which 
the conclusion of Proposition~\ref{horseshoe} holds with~$\varepsilon_0$
replaced by~$\varepsilon_0 /2$.
Since $H_n \cap B(x_0,\rho) \subset \bigcup_{W\in\mathcal Q_n} W$,
$$\log\left|H_n \cap B(x_0,\rho) \right|
\leq
\log \left( \sum_{W\in\mathcal Q_n}|W| \right).$$
Let $\mu\in\mathcal M(f^q)$ be as in Lemma~\ref{horse} applied to~$B =
B(x_0, 2\rho)$, $\Delta = e^{\frac{\varepsilon_0}{2} n}$, and the pull-backs~$L_1,\ldots,L_t$
of~$B(x_0, 2\rho)$ by~$f^q$.
Proposition~\ref{horseshoe}(c) yields $\int \varphi_j d \mu > \alpha_j
- \varepsilon_0$ for every $j\in\{1,2,\ldots,l\}$.
On the other hand, using $|B(x_0, \rho)| \le 1$ and Proposition~\ref{horseshoe}(b), for every large $n$ we have  
$$ \log \left(\sum_{W\in\mathcal Q_n}|W|\right)
\leq
\log \left(\sum_{i=1}^t|L_i| \right) + \frac{\varepsilon_0}{2} n
\leq qF(\mu) + \varepsilon_0 n. $$
Since $q\geq n$ and $F(\mu)\leq0$  from Proposition \ref{lower-prop}, 
we have
\begin{align*}
\frac{1}{n}\log \left(\sum_{W\in\mathcal Q_n}|W|\right)
\leq
\frac{q}{n} F(\mu) + \varepsilon_0
\leq
F(\mu) + \varepsilon_0.
\end{align*}
This yields the desired inequality.
\end{proof}

\begin{prop}\label{upper}
Let $f\colon X\to X$ have H{\"o}lder continuous derivative and only
non-flat critical points.
Assume~$f$ is topologically exact.
Let $\varepsilon_0 >0$, let $l\geq1$ be an integer,
let $\varphi_1,\ldots,\varphi_l\colon X\to\mathbb R$ be
continuous functions, and let $\alpha_1,\ldots,\alpha_l\in\mathbb R$.
Then
\begin{multline*}
\limsup_{n\to\infty}\frac{1}{n} \log \left|\left\{x\in X\colon \frac{1}{n}S_n\varphi_j(x)\geq\alpha_j\ \text{ for every $j\in\{1,\ldots,l\}$}\right\}\right|\\
\\ \leq
\sup\left\{F(\mu)\colon\text{$\mu\in\mathcal M(f)$ and $\int\!\varphi_jd\mu>\alpha_j- \varepsilon_0 $ for every $j\in\{1,\ldots,l\}$}\right\}+ \varepsilon_0 .
\end{multline*}
\end{prop}

\begin{remark}
Since the Lyapunov exponent is not lower semi-continuous in general, it is not possible 
to let $\varepsilon_0 =0$ in the inequality in Proposition \ref{upper}.
\end{remark}

\begin{proof}[Proof of Proposition \ref{upper}]
Let $\varepsilon_0>0$, $l \ge 1$, $\varphi_1,\ldots,\varphi_l$,
and~$\alpha_1,\ldots,\alpha_l$ be as in the statement of the proposition.
Let~$\rho>0$ denote the constant for which the conclusion of
Proposition~\ref{upper0} holds with~$\varepsilon_0$ replaced by $\varepsilon_0 /2$.
Fix a large integer $M\geq1$ with
$f^M(B(x_0,\rho))=X$. 
Since each $\varphi_j$ is bounded, for sufficiently large $n$ we have
\begin{multline*}
\left\{x\in f^M(B(x_0,\rho))\colon
  \frac{1}{n}S_n\varphi_j(x)\geq\alpha_j\ \text{ for every
    $j\in\{1,\ldots,l\}$}\right\}
\\ \subset
f^M\left\{x\in B(x_0,\rho)\colon
  \frac{1}{n}S_n\varphi_j(x)\geq\alpha_j - \frac{\varepsilon_0}{2} \ \text{ for every
    $j\in\{1,\ldots,l\}$}\right\},
\end{multline*}
and therefore
\begin{align*}
& \frac{1}{n}\log\left|\left\{x\in X\colon
    \frac{1}{n}S_n\varphi_j(x)\geq\alpha_j\ \text{ for every
      $j\in\{1,\ldots,l\}$}\right\}\right|
\\ & \leq
\frac{1}{n} \log \left[ \left( \sup_X |Df| \right)^M \cdot \left|\left\{x\in B(x_0,\rho)\colon
    \frac{1}{n}S_n\varphi_j(x)\geq\alpha_j - \frac{\varepsilon_0}{2} \ \text{ for every
      $j\in\{1,\ldots,l\}$}\right\}\right| \right]
 \\ &\leq
\frac{1}{n} \log\left|\left\{x\in B(x_0,\rho)\colon
    \frac{1}{n}S_n\varphi_j(x)\geq\alpha_j - \frac{\varepsilon_0}{2} \ \text{ for every
      $j\in\{1,\ldots,l\}$}\right\}\right|+\frac{\varepsilon_0}{2}.  
\end{align*}
We use Proposition~\ref{upper0} with~$\alpha_j$ replaced by~$\alpha_j -  \varepsilon_0 /2 $ for
every~$j\in\{1, \ldots, l\}$.
For each sufficiently large~$n$ there exists $\mu\in\mathcal M(f)$ such that
$\int \varphi_jd\mu>\alpha_j -  \varepsilon_0$ for every $j\in\{1,\ldots,l\}$, and 
$$ \frac{1}{n} \log\left|\left\{x\in B(x_0,\rho)\colon
    \frac{1}{n}S_n\varphi_j(x)\geq\alpha_j- \frac{\varepsilon_0}{2}\
    \text{ for every $j\in\{1,\ldots,l\}$}\right\}\right|
\leq
F(\mu) + \frac{\varepsilon_0}{2}.$$
Letting $n\to\infty$ we obtain the proposition.
\end{proof}

\subsection{End of the large deviations upper bound}\label{end}
Let $f\colon X\to X$ have H{\"o}lder continuous derivative and only
non-flat critical points, and assume it is topologically exact.
Let~$\mathcal K$ be a closed subset of $\mathcal M$, and let~$\mathcal G$ be an arbitrary open set containing~$\mathcal K$.
Since~$\mathcal K$ is compact, 
one can choose a finite collection $\mathcal C_1, \ldots, 
\mathcal C_r$ of closed sets
such that $\mathcal K \subset \bigcup_{k=1}^r \mathcal C_k \subset \mathcal G $
and such that each of them has the form
$$\mathcal C_k=\left\{\mu\in\mathcal M\colon \int\!\varphi_j d\mu\geq\alpha_j\text{ for every $j\in\{1,\ldots,p\}$}\right\},$$ where $p\geq1$ is an integer, each $\varphi_j\colon X\to\mathbb R$ is a continuous function and 
$\alpha_j\in\mathbb R$.
For each $k\in\{1,2,\ldots,r\}$ and $\varepsilon_0>0$ define an open neighborhood
$\mathcal C_k(\varepsilon_0 )$ of $\mathcal C_k $
by replacing $\int\varphi_j d\nu\geq \alpha_j$ in the definition of $\mathcal C_k$ by $\int\varphi_j d\nu> \alpha_j-
\varepsilon_0$.
From Proposition \ref{upper}, for every $\varepsilon_0 >0$ and every $k\in\{1,2,\ldots,r\}$,
$$\limsup_{n\to\infty}\frac{1}{n}\log\left|\left\{x\in X\colon
\delta_{x}^n\in\mathcal C_k\right\}\right|\leq \sup_{ \mathcal C_k(\varepsilon)}F+ \varepsilon_0.$$
Since
$\bigcup_{k=1}^r \mathcal C_k (\varepsilon_0)\subset \mathcal G$
for $\varepsilon_0 >0$ small enough,
using the previous inequality for each $k\in\{1,2,\ldots,r\}$ gives
\begin{align*}
\limsup_{n\to\infty}\frac{1}{n}\log\left|\left\{x\in X\colon \delta_x^n\in\mathcal K\right\}\right|
&\le \limsup_{n\to\infty}\frac{1}{n}
\log\left|\left\{x\in X\colon\delta_x^n \in\bigcup_{k=1}^r {\mathcal C}_k \right\}\right|\\
&\leq
\max_{k\in\{1,2,\ldots,r\}} \limsup_{n\to\infty}\frac{1}{n}\log
\left|\left\{x\in X\colon\delta_x^n \in {\mathcal C}_k \right\}\right|\\
& \le \max_{k\in\{1,2,\ldots,r\}}  
\sup_{\mathcal C_k (\varepsilon_0)} F  +{\varepsilon_0}
\\ & \le
\sup_{\mathcal G}F+{\varepsilon_0}.
\end{align*}
Letting~${\varepsilon_0} \to 0$ we obtain
$$\limsup_{n\to\infty}\frac{1}{n}\log\left|\left\{x\in X\colon \delta_x^n\in\mathcal K\right\}\right|
\le \sup_{\mathcal G}F.$$
Since $\mathcal G$ is an arbitrary open set containing 
 $\mathcal K$, it follows that
\begin{align*}
\limsup_{n\to\infty}\frac{1}{n}\log\left|\left\{x\in X\colon  \delta_{x}^n\in\mathcal K\right\}\right|
\leq
\inf_{\mathcal G \supset \mathcal{K}} \sup_{ \mathcal G } F
=
\inf_{\mathcal G \supset \mathcal{K}} \sup_{  \mathcal G } (-I)
=
-\inf_{\mathcal K}I.\end{align*}
The last equality is due to the upper semi-continuity of $-I$. 
\qed

\appendix
\section{Rate functions for Hofbauer-Keller maps}
Let $f_a\colon X\to X \enspace (0< a \leq 4)$ be the quadratic map $f_a(x)=ax(1-x)$.
Let $c=1/2$ and put $X_a=[ f_a^2 (c), f_a(c)]$. Notice that $f_a(X_a)=X_a$.
Denote by $\mathcal M_a$ the space of Borel probability measures on $X_a$ endowed with the
weak* topology, and by $\mathcal M_a(f_a)$ the set of elements of $\mathcal M_a$ which are $f_a|_{X_a}$-invariant.

By \cite{GraSwi97,Lyu02}, for Lebesgue almost every $a\in (0,4]$ there exists a unique
physical measure of $f_a$. 
Based on the kneading theory, Hofbauer \& Keller \cite{HofKel90, HofKel95} constructed various examples of quadratic maps with unexpected properties.
One of them is the following.

\begin{theorem}[\cite{HofKel95}, Propositions~1 and 2]
  \label{HK}
There is a uncountable set $A\subset (0, 4)$ such that if $a\in A$ then
$f_a$ is non-renormalzable and
there are sequences $\{n_i\}_i$, $\{m_i\}_i$ of positive integers
 with $n_i <m_i < n_{i+1}$ for each $i$
such that the following holds:
\begin{enumerate}
 \item[(a)] $\displaystyle \left| \int\! \varphi \, d\delta_{x}^{n_i}  -  \int\! \varphi \, d\delta_{  {c} }^{m_i}  \right| \to 0
 \enspace (i\to\infty)$ for Lebesgue almost every $x\in X_a$ and each continuous
 $\varphi \colon X_a\to\mathbb R$;
 \item[(b)] if $z\in X_a$ and $p\geq1$ are such that $f^p(z)=z$,
 then $\delta_{z}^{p}$ is an weak*-accumulation point 
 of the sequence $\{  \delta_{c}^{m_i}\}_{i\geq1}$.
  \end{enumerate}
\end{theorem}

  In particular, if $a\in A$ then there is no physical measure of $f_a$. Hence,
   the law of large numbers does not hold for the Birkhoff sum $\varphi+\varphi\circ f_a+\cdots+\varphi\circ f_a^{n-1}$
  of a continuous function $\varphi\colon X_a\to\mathbb R$.
Nevertheless, $f_a|_{X_a}$ satisfies the hypotheses of the Main Theorem
and hence the LDP holds. The rate function is identically zero on its effective domain.

\begin{theorem}\label{HKrate}
Let $A$ be the set as in Theorem~\ref{HK}.
If $a\in A$ then the large deviations rate function of $f_a|_{X_a}$ is identically zero  
on $\mathcal M_{a} (f_a)$.
\end{theorem}

\begin{proof}
Let $a\in A$ and $\mu\in \mathcal M_a (f_a)$.
Let $\mathcal U$ be an arbitrary open set containing $\mu$.
Take $l\geq1$, continuous functions $\varphi_1,\ldots,\varphi_l\colon X_a\to\mathbb R$, $\varepsilon >0$ such that
$\mu \in \mathcal C \subset \mathcal U$, where 
$$\displaystyle \mathcal C= 
\left\{ \nu\in\mathcal M_a\colon  \left|   \int\! \varphi_j  d\nu  -  \int\! \varphi_j  d\mu  \right| \leq \varepsilon
\ \text{ for every $j\in\{1,\ldots,l\}$}  \right\} .$$
Since $f_a$ is non-renormalizable, its restriction to $X_a$ is topologically exact and has the specification property.
Hence $\mu$ is weak*-approximated by another supported on a periodic orbit
 \cite[Theorem~1]{Sig74} and
 there exist $z\in X_a$ and $p\geq1$ such that
$f_a^p(z)=z$ and
\begin{equation*}
 \left|   \int\! \varphi_j  d\delta_z^p  -  \int\! \varphi_j  d\mu  \right| \leq\frac{\varepsilon}{3}
\ \text{ for every $j\in\{1,\ldots,l\}$}.
\end{equation*}
From Theorem~\ref{HK} there are increasing sequences 
$\{n_i\}_i$, $\{m_i\}_i$ 
of positive integers for which the following holds:
\begin{equation*}
 \left| \left\{ x\in X_a\colon  \left|  \int\! \varphi_j  d\delta_{x}^{n_{i}}  -  \int\! \varphi_j  d\delta_{c}^{m_{i}}  \right| \leq 
 \frac{\varepsilon}{3}  \ \text{ for every $j\in\{1,\ldots,l\}$} \right\} \right| \geq \frac{1}{2};
\end{equation*}
\begin{equation*}   
 \left|   \int\! \varphi_j  d\delta_{c}^{m_{i}}  -  \int\! \varphi_j  d\delta_z^p  \right| \leq \frac{\varepsilon}{3}
\ \text{ for every $j\in\{1,\ldots,l\}$}.
\end{equation*}
Combining these three inequalities yields
\begin{equation}\label{hkeq4}
\frac{1}{2} \leq 
\left|\left\{ x\in X_a\colon \left|  \int\! \varphi_j  d\delta_{x}^{n_{i}}  -  \int\! \varphi_j  d\mu  \right| \leq \varepsilon 
\ \text{ for every $j\in\{1,\ldots,l\}$} \right\} \right| \leq |X_a|\le 1.
\end{equation}
Denote by $I_a$ the large deviations rate function of $f_a|_{X_a}$.
Then 
\begin{align*}
0 &\leq \limsup_{n\to \infty} \frac{1}{n}\log 
\left| \left\{ x\in X_a:  \delta_{x}^{n} \in \mathcal C   \right\} \right| \leq - \inf_{\mathcal C} I_a 
\leq  - \inf_{ \mathcal U} I_a
\leq 0.
\end{align*}
The first inequality is from \eqref{hkeq4} and the second from the Main Theorem.
Hence $\displaystyle  \inf_{ \mathcal U} I_a = 0$.
Since $\mathcal U$ is an arbitrary open set containing $\mu$ and $I_a$ is lower semi-continuous,
$I_a(\mu)=0$.
\end{proof}

\subsection*{Acknowledgments}
We would like to thank Micha{\l} Misiurewicz for his help with references,
Bing Gao, Gerhard Keller and Masato Tsujii for fruitful discussions, and the anonymous referees for their healthy criticism that helped us improve the exposition in the introduction.
The first-named author is partially supported by the Grant-in-Aid for Scientific Research (C) of the JSPS 16K05179.
The second-named author is partially supported by FONDECYT grant 1141091 and NSF Grant DMS-1700291.
The last-named author is partially supported by the Grant-in-Aid for Young Scientists (A) of the JSPS 15H05435
and the Grant-in-Aid for Scientific Research (B) of the JSPS 16KT0021.

\bibliographystyle{abbrv}

\begin{thebibliography}{10}

\bibitem{AviLyudMe03}
A.~Avila, M.~Lyubich, and W.~de~Melo.
\newblock Regular or stochastic dynamics in real analytic families of unimodal
  maps.
\newblock {\em Invent. Math.}, 154(3):451--550, 2003.

\bibitem{BenCar85}
M.~Benedicks and L.~Carleson.
\newblock On iterations of {$1-ax\sp 2$} on {$(-1,1)$}.
\newblock {\em Ann. of Math. (2)}, 122(1):1--25, 1985.

\bibitem{BloOve04a}
A.~Blokh and L.~Oversteegen.
\newblock Backward stability for polynomial maps with locally connected {J}ulia
  sets.
\newblock {\em Trans. Amer. Math. Soc.}, 356(1):119--133, 2004.

\bibitem{Bow75}
R.~Bowen.
\newblock {\em Equilibrium states and the ergodic theory of {A}nosov
  diffeomorphisms}.
\newblock Lecture Notes in Mathematics, 470. Springer-Verlag, Berlin,
  1975.

\bibitem{BruKel98}
H.~Bruin and G.~Keller.
\newblock Equilibrium states for {$S$}-unimodal maps.
\newblock {\em Ergodic Theory Dynam. Systems}, 18(4):765--789, 1998.

\bibitem{CaiLi09}
H.~Cai and S.~Li.
\newblock Distortion of interval maps and applications.
\newblock {\em Nonlinearity}, 22(10):2353--2363, 2009.

\bibitem{Chu11}
Y.~M. Chung.
\newblock Large deviations on {M}arkov towers.
\newblock {\em Nonlinearity}, 24(4):1229--1252, 2011.

\bibitem{ChuTak12}
Y.~M. Chung and H.~Takahasi.
\newblock Large deviation principle for {B}enedicks-{C}arleson quadratic maps.
\newblock {\em Comm. Math. Phys.}, 315(3):803--826, 2012.

\bibitem{ChuTak14}
Y.~M. Chung and H.~Takahasi.
\newblock Multifractal formalism for {B}enedicks-{C}arleson quadratic maps.
\newblock {\em Ergodic Theory Dynam. Systems}, 34(4):1116--1141, 2014.

\bibitem{ChuTak17}
Y.~M. Chung and H.~Takahasi.
\newblock {Large deviation principle for $S$-unimodal maps with flat critical
  point}.
\newblock arXiv:1708.03695v2, 2017.


\bibitem{ColEck83}
P.~Collet and J.-P. Eckmann.
\newblock Positive {L}iapunov exponents and absolute continuity for maps of the
  interval.
\newblock {\em Ergodic Theory Dynam. Systems}, 3(1):13--46, 1983.

\bibitem{ComRiv11}
H.~Comman and J.~Rivera-Letelier.
\newblock Large deviation principles for non-uniformly hyperbolic rational
  maps.
\newblock {\em Ergodic Theory Dynam. Systems}, 31(2):321--349, 2011.

\bibitem{dMevSt93}
W.~de~Melo and S.~van Strien.
\newblock {\em One-dimensional dynamics}. 
Ergebnisse der Mathematik und ihrer Grenzgebiete (3), 25.
\newblock Springer-Verlag, Berlin, 1993.

\bibitem{DemZei98}
A.~Dembo and O.~Zeitouni.
\newblock {\em Large deviations techniques and applications}. Second edition.
Applications of Mathematics (New York), 38.
\newblock Springer-Verlag, New York,  1998.

\bibitem{Den96}
M.~Denker.
\newblock Probability theory for rational maps.
\newblock In {\em Probability theory and mathematical statistics ({S}t.\
  {P}etersburg, 1993)}, pages 29--40. Gordon and Breach, Amsterdam, 1996.

\bibitem{Dob14}
N.~Dobbs.
\newblock On cusps and flat tops.
\newblock {\em Ann. Inst. Fourier (Grenoble)}, 64(2):571--605, 2014.

\bibitem{DonVar75}
M.~D. Donsker and S.~R.~S. Varadhan.
\newblock Asymptotic evaluation of certain {M}arkov process expectations for
  large time. {I}. {II}.
\newblock {\em Comm. Pure Appl. Math.}, 28:1--47, 1975; ibid. 28:279--301, 1975.

\bibitem{Ell85}
R.~S. Ellis.
\newblock {\em Entropy, large deviations, and statistical mechanics}, 
Grundlehren der Mathematischen Wissenschaften, 271.
\newblock Springer-Verlag, New York, 1985.

\bibitem{GraSwi97}
J.~Graczyk and G.~{\'S}wiatek.
\newblock Generic hyperbolicity in the logistic family.
\newblock {\em Ann. of Math. (2)}, 146(1):1--52, 1997.

\bibitem{Gri93}
J.~Grigull.
\newblock {\em Gro{\ss}e {A}bweichungen und {F}luktuationen f{\"u}r
  {G}leichgewichtsma{\ss}e rationaler {A}bbildungen}.
\newblock PhD thesis, 1993.

\bibitem{HofKel90}
F.~Hofbauer and G.~Keller.
\newblock Quadratic maps without asymptotic measure.
\newblock {\em Comm. Math. Phys.}, 127(2):319--337, 1990.

\bibitem{HofKel95}
F.~Hofbauer and G.~Keller.
\newblock Quadratic maps with maximal oscillation.
\newblock In {\em Algorithms, fractals, and dynamics ({O}kayama/{K}yoto,
  1992)}, pages 89--94. Plenum, New York, 1995.

\bibitem{Jak81}
M.~V. Jakobson.
\newblock Absolutely continuous invariant measures for one-parameter families
  of one-dimensional maps.
\newblock {\em Comm. Math. Phys.}, 81(1):39--88, 1981.

\bibitem{KatHas95}
A.~Katok and B.~Hasselblatt.
\newblock {\em Introduction to the modern theory of dynamical systems}.
With a supplementary chapter by Katok and Leonardo Mendoza.
Encyclopedia of Mathematics and its Applications, 54.
\newblock Cambridge University Press, Cambridge, 1995.

\bibitem{Kel98}
G.~Keller.
\newblock {\em Equilibrium states in ergodic theory}.  London
  Mathematical Society Student Texts, 42.
\newblock Cambridge University Press, Cambridge, 1998.

\bibitem{KelNow92}
G.~Keller and T.~Nowicki.
\newblock Spectral theory, zeta functions and the distribution of periodic
  points for {C}ollet-{E}ckmann maps.
\newblock {\em Comm. Math. Phys.}, 149(1):31--69, 1992.

\bibitem{KelNow95}
G.~Keller and T.~Nowicki.
\newblock Fibonacci maps re(al)visited.
\newblock {\em Ergodic Theory Dynam. Systems}, 15(1):99--120, 1995.

\bibitem{Kif90}
Y.~Kifer.
\newblock Large deviations in dynamical systems and stochastic processes.
\newblock {\em Trans. Amer. Math. Soc.}, 321(2):505--524, 1990.

\bibitem{Lev98a}
G.~Levin.
\newblock On backward stability of holomorphic dynamical systems.
\newblock {\em Fund. Math.}, 158(2):97--107, 1998.

\bibitem{Li16}
H.~Li.
\newblock Large deviation principles of one-dimensional maps for {H}\"older
  continuous potentials.
\newblock {\em Ergodic Theory Dynam. Systems}, 36(1):127--141, 2016.

\bibitem{Lyu02}
M.~Lyubich.
\newblock Almost every real quadratic map is either regular or stochastic.
\newblock {\em Ann. of Math. (2)}, 156(1):1--78, 2002.

\bibitem{LyuMil93}
M.~Lyubich and J.~Milnor.
\newblock The {F}ibonacci unimodal map.
\newblock {\em J. Amer. Math. Soc.}, 6(2):425--457, 1993.

\bibitem{MelNic08}
I.~Melbourne and M.~Nicol.
\newblock Large deviations for nonuniformly hyperbolic systems.
\newblock {\em Trans. Amer. Math. Soc.}, 360(12):6661--6676, 2008.

\bibitem{MisSzl80}
M.~Misiurewicz and W.~Szlenk.
\newblock Entropy of piecewise monotone mappings.
\newblock {\em Studia Math.}, 67(1):45--63, 1980.

\bibitem{OrePel89}
S.~Orey and S.~Pelikan.
\newblock Deviations of trajectory averages and the defect in {P}esin's formula
  for {A}nosov diffeomorphisms.
\newblock {\em Trans. Amer. Math. Soc.}, 315(2):741--753, 1989.

\bibitem{PolShaYur98} 
M.~Pollicott, R.~Sharp and M.~Yuri. 
\newblock Large deviations for maps with indifferent fixed points. 
\newblock {\em Nonlinearity}, 11(4):1173--1184, 1998.

\bibitem{PrzRiv11}
F.~Przytycki and J.~Rivera-Letelier.
\newblock Nice inducing schemes and the thermodynamics of rational maps.
\newblock {\em Comm. Math. Phys.}, 301(3):661--707, 2011.

\bibitem{PrzRiv19}
F.~{Przytycki} and J.~{Rivera-Letelier}.
\newblock Geometric pressure for multimodal maps of the interval.
\newblock {\em Mem. Amer. Math. Soc.}, 259(1246):v+81, 2019.

\bibitem{PrzRivSmi03}
F.~Przytycki, J.~Rivera-Letelier and S.~Smirnov.
\newblock Equivalence and topological invariance of conditions for non-uniform
  hyperbolicity in the iteration of rational maps.
\newblock {\em Invent. Math.}, 151(1):29--63, 2003.

\bibitem{PrzUrb10}
F.~Przytycki and M.~Urba{\'n}ski.
\newblock {\em Conformal fractals: ergodic theory methods}. 
  London Mathematical Society Lecture Note Series, 371.
\newblock Cambridge University Press, Cambridge, 2010.


\bibitem{ReyYou08}
L.~Rey-Bellet and L.-S. Young.
\newblock Large deviations in non-uniformly hyperbolic dynamical systems.
\newblock {\em Ergodic Theory Dynam. Systems}, 28(2):587--612, 2008.

\bibitem{Riv1204}
J.~{Rivera-Letelier}.
\newblock {Asymptotic expansion of smooth interval maps}.
\newblock arXiv:1204.3071v2, 2012.

\bibitem{Rue76}
D.~Ruelle.
\newblock A measure associated with axiom-{A} attractors.
\newblock {\em Amer. J. Math.}, 98(3):619--654, 1976.

\bibitem{Rue78}
D.~Ruelle.
\newblock An inequality for the entropy of differentiable maps.
\newblock {\em Bol. Soc. Brasil. Mat.}, 9(1):83--87, 1978.

\bibitem{Rue17}
S.~Ruette.
{\it Chaos on the Interval.} University Lecture Series, 67. American Mathematical Society,  Providence, RI, 2017.

\bibitem{Sig74}
K.~Sigmund.
\newblock On dynamical systems with the specification property.
\newblock {\em Trans. Amer. Math. Soc.}, 190:285--299, 1974.

\bibitem{Sin72}
J.~G. Sina{\u\i}.
\newblock Gibbs measures in ergodic theory.
\newblock {\em Uspehi Mat. Nauk}, 27(4(166)):21--64, 1972.

\bibitem{Tak84}
Y.~Takahashi.
\newblock Entropy functional (free energy) for dynamical systems and their
  random perturbations.
\newblock In {\em Stochastic analysis ({K}atata/{K}yoto, 1982)}, North-Holland Math. Library, 32, pages 437--467. North-Holland, Amsterdam, 1984. 

\bibitem{Tak87}
Y.~Takahashi.
\newblock Asymptotic behaviours of measures of small tubes: entropy,
  {L}iapunov's exponent and large deviation.
\newblock In {\em Dynamical systems and applications ({K}yoto, 1987)}, 
 World Sci. Adv. Ser. Dynam. Systems, 5, pages 1--21. World Sci.
  Publishing, Singapore, 1987.

\bibitem{You90}
L.-S. Young.
\newblock Large deviations in dynamical systems.
\newblock {\em Trans. Amer. Math. Soc.}, 318(2):525--543, 1990.

\end{thebibliography}

\end{document}